\documentclass[oneside,english]{amsart}

\usepackage{amsmath,amssymb,amscd,amsthm}
\usepackage{amsfonts, tikz-cd}
\usepackage{xypic}

\usepackage{accents}
\usepackage{graphicx}
\usepackage{verbatim}
\usepackage{tikz}
\usepackage{enumitem}
\usepackage{ mathrsfs }
\usepackage{subcaption}
\usepackage{mwe}

\usepackage{microtype}

\usepackage{booktabs}

\usepackage{graphicx}
\usepackage{color}
\definecolor{dark-red}{rgb}{0.6,0,0}
\definecolor{dark-green}{rgb}{0,0.6,0}
\definecolor{dark-blue}{rgb}{0,0,0.6}

\newlength{\dhatheight}

\newtheorem{theorem}{Theorem}[subsection]

\newtheorem{proposition}[theorem]{Proposition}
\newtheorem{lemma}[theorem]{Lemma}
\newtheorem{remark}[theorem]{Remark}
\newtheorem{corollary}[theorem]{Corollary}
\newtheorem{definition}[theorem]{Definition}

\newtheorem{question}[theorem]{Question}

\newtheorem{theoremss}{Theorem}[section]
\newtheorem{conjecturess}[theoremss]{Conjecture}

\newtheorem{lemmass}[theoremss]{Lemma}

\newtheorem{definitionss}[theoremss]{Definition}

\newtheorem{questionss}[theoremss]{Question}

\newcommand{\Z}{{\mathbb Z}}
\newcommand{\Q}{{\mathbb Q}}

\newcommand{\p}{\partial}

\newlist{pcases}{enumerate}{1}
\setlist[pcases]{
  label={\em{Case~\arabic*:}}\protect\thiscase.~,
  ref=\arabic*,
  align=left,
  labelsep=0pt,
  leftmargin=0pt,
  labelwidth=0pt,
  parsep=0pt
}
\newcommand{\case}[1][]{%
  \if\relax\detokenize{#1}\relax
    \def\thiscase{}%
  \else
    \def\thiscase{~#1}%
  \fi
  \item
}

\newlist{subpcases}{enumerate}{1}
\setlist[subpcases]{
  label={\em{Subcase~\Alph*:}}\protect\thiscase~,
  ref=\Alph*,
  align=left,
  labelsep=0pt,
  leftmargin=0pt,
  labelwidth=0pt,
  parsep=0pt
}
\newcommand{\subcase}[1][]{%
  \if\relax\detokenize{#1}\relax
    \def\thiscase{}%
  \else
    \def\thiscase{~#1}%
  \fi
  \item
}

\usepackage[%
pdfauthor={Nicholas John Owad}, pdftitle={Bridge spectra of cables of 2-bridge knots}, pdfsubject={Thesis}, pdfkeywords={LaTeX, Thesis,
University of Nebrska, Test}, linkcolor=black,
pagecolor=dark-green, citecolor=black, urlcolor=dark-red,
colorlinks=true, backref,
plainpages=false,% This helps to fix the issue with hyperref with page numbering
pdfpagelabels% This helps to fix the issue with hyperref with page numbering
]{hyperref}

\usepackage{memhfixc}
\thanks{  The author was partially supported by the National Science Foundation under NSF Grant DMS-1313559, P.I. Susan Hermiller}

\begin{document}
%% Start formating the first few special pages
%% frontmatter is needed to set the page numbering correctly

\title{Bridge spectra of cables of 2-bridge knots}
\author{Nicholas Owad}

\maketitle
%% You have a maximum of 350, which includes your title, name, etc.
\begin{abstract}
We compute the bridge spectra of cables of 2-bridge knots.  We also give some results about bridge spectra and distance of Montesinos knots.
\end{abstract}
%
%
%
%%% Optional

%% The ToC is required
%% Uncomment these if need be

%% The ToC is required
\tableofcontents
%% Uncomment these if need be
%\listoffigures
%\listoftables

%% Thesis goes here *********************************************************************************

\section{Introduction}\label{sec:intro}

%History and motivation~\cite{belkbux} possibly referencing definitions
%in later subsections~\ref{def:filling},

%%%%

Knot theory is the study of knotted curves in space.  In 1954, Horst Schubert defined an invariant called the bridge number.  This is defined as the smallest number of local maxima over all projections of the knot.  He proved a very important property of this invariant: that it is additive  minus one under connect sum. Using the fact that only the unknot has a bridge number of one, he shows that knots do not form a group under the operation of connect sum, i.e. there are no ``inverses" for the operation.

The bridge spectrum of a knot is a generalization of Schubert's bridge number and we explore some of its properties here.  In this paper, we compute the bridge spectra of a variety of classes of knots.  The classes of knots for which the bridge spectrum is known is relatively short: iterated torus knots (which include torus knots), 2-bridge knots, high distance knots, and partial results for twisted torus knots.  We add to this list all cables of 2-bridge knots and a class of generalized Montesinos knots, which includes all pretzel knots that satisfy a condition on their tangled regions.  

The bridge spectrum of a knot $K$ is a strictly decreasing list of nonnegative integers first defined by Doll \cite{doll} and Morimoto and Sakuma \cite{MS}.  This list is obtained from $(g,b)$ -splittings of a knot in $S^3$.  A {\em $(g,b)$-splitting} of a knot $K$ is a Heegaard splitting of $S^3=V_1 \cup_\Sigma V_2$ such that the genus of $V_i$ is $g$, $\Sigma$ intersects $K$ transverselly, and $V_i \cap K$ is a collection of $b$ trivial arcs for $i=1,2$. The {\em genus $g$ bridge number}  $b_g(K)$ is the minimum number $b$ for which a $(g,b)$-splitting exists.  The {\em bridge spectrum} for $K$, which we denote ${\mathbf b}(K)$, is the list 

$$(b_0(K),b_1(K),b_2(K),\ldots ).$$

Note that $b_0(K)$ is the classical bridge number, $b(K)$, first defined by Schubert \cite{Schubert}.  The fact that bridge spectrum is strictly decreasing comes from the process called {\em meridional stabilization} where, given a $(g,b)$-splitting, one can create a $(g+1,b-1)$-splitting. Thus, every bridge spectrum is bounded above componentwise by the following sequence, which is called a {\em stair-step} spectrum,

 $$\left(b_{0}\left(K\right),b_{0}\left(K\right)-1,b_{0}\left(K\right)-2,\ldots,2,1,0\right).$$  

%%%%%%%%
%Stairstep begin

An invariant of a bridge surface of a knot is the distance, where we are considering distance in terms of the curve complex of the bridge surface (see page \pageref{def:bridgesurface} and Definition \ref{def:dist} for relevant terms). Tomova showed in \cite{Tomova}, with results of Bachman and Schleimer from \cite{BS}, the following theorem, as stated by Zupan in \cite{Z1}:

\smallskip

\begin{theorem}\label{thm:dist}{\cite{Tomova}}
Suppose $K$ is a knot in $S^3$ with a $(0,b)$-bridge sphere $\Sigma$ of sufficiently high distance (with respect to $b$). Then any $(g',b')$-bridge surface $\Sigma ' $  satisfying $b' = b_{g'} (K)$ is the result of meridional stabilizations performed on $\Sigma$. Thus
 $${\mathbf b}(K)=\left(b_{0}\left(K\right),b_{0}\left(K\right)-1,b_{0}\left(K\right)-2,\ldots,2,1,0\right).$$  
 
\end{theorem}

A natural question to ask is the following:

\begin{question}
Given a knot, $K$, if the bridge spectrum of $K$ is stair-step, does $K$ necessarily have a $(0,b)$-bridge sphere that is of high distance?
\end{question}

We answer this question in the negative in subsection \ref{sec:BSMontandPret} by showing the two results that follow.

Lustig and Moriah \cite{LM} define generalized Montesinos knots, which include Montesinsos knots and pretzel knots.  Generalized Montesinos knots, see Defintion \ref{def:genMont}, are defined with a set of pairs of integers, $\{ (\beta_{i,j},\alpha_{i,j}) \}_{i,j}$; when $\gcd\{ \alpha_{i,j}\} \not = 1$, we show that these generalized Montesinos knots have the property $t(K)+1=b(K)=b_0(K)$, where $t(K)$ is the tunnel number, see subsection \ref{sec:tunnel}.  For a $(g,b)$-splitting, with $b>0$, we have $t(K) \leq g+b-1$, which comes to us from Morimoto, Sakuma, and Yokota \cite{MSY}.  From these two facts, we can quickly conclude that generalized Montesinos knots with $\gcd\{ \alpha_{i,j}\} \not = 1$ have stair-step bridge spectra.  

\noindent{\textbf{Corollary}~\ref{cor:bsp}.} 
{\em 
Given a pretzel knot $K_n=K_n(p_1,\ldots,p_n)$ with $\gcd(p_1,\ldots,p_n)\not = 1$, then the primitive bridge spectrum is stair-step, i.e. ${\mathbf {\hat b}}(K_n)=(n,n-1,\ldots,2,1,0)$ and ${\mathbf b}(K_n)=(n,n-1,\ldots,3,2,0)$.
}

\smallskip

\noindent{\textbf{Proposition}~\ref{prop:dist1}.} 
{\em 
If $K_n(p_1,\ldots,p_n)$ is a pretzel knot with $n\geq 4$,  then there is a $(0,n)$-bridge surface for $K_n$ with $d(P,K_n)=1$.
}

\smallskip

By these two results, we see that stair-step bridge spectrum does not imply high distance.

Theorem \ref{thm:dist} tells us that a generic knot has stair-step bridge spectrum. It is said there is a {\em gap} at index $g$ in the bridge spectrum if $b_g(K)<b_{g-1}(K)-1$.  Most work regarding bridge spectra focuses on finding bridge spectra with gaps, see \cite{Z1}, \cite{BTZ}.  This includes Theorem \ref{thm:BSofC2b} and is the topic of most of the conjectures in section \ref{ch:generalize}.  The rest of this subsection is devoted to the statement Theorem \ref{thm:BSofC2b}.

%stairstep end
%%%%%%%

We are interested in the behavior of the bridge spectrum under cabling, a special case of taking a satellite of a knot.  In the same paper that Schubert \cite{Schubert} defined bridge number, he proved the following, which we will make use of here.  Schultens gives a modern proof of this result in \cite{Schultens}.  See subsection \ref{sec:cable}.

\begin{theorem}\label{thm:bridgesat}{{\cite{Schubert},\cite{Schultens}}}
Let $K$ be a satellite knot with companion $J$ and pattern of index $n$. Then $b_0(K)\geq n\cdot b_0(J)$.
\end{theorem}

In section \ref{sec:proof}  we show:
\smallskip

\noindent{\textbf{Theorem}~\ref{thm:BSofC2b}.} 
{\em 
Let $K_{p/q}$ be a non-torus 2-bridge knot and $T_{m,n}$ an $(m,n)$-torus knot. If $K:={\mathsf {cable}}(T_{m,n}, K_{p/q})$ is a cable of  $K_{p/q}$ by $T_{m,n}$, then the bridge spectrum of $K$ is ${\mathbf b}(K)=(2m,m,0)$.
}
\smallskip

In section \ref{sec:notation} we give the relevant background and notation needed for the the proofs of section \ref{ch:MontAndPret} and Theorem \ref{thm:BSofC2b}.  In section \ref{sec:Alex} we give some necessary results and lemmas for the proof of the Theorem \ref{thm:BSofC2b}. Section \ref{ch:MontAndPret} has proofs of Corollary~\ref{cor:bsp} and Proposition~\ref{prop:dist1} and details about a class of knots which are stair-step but not of high distance.  In section \ref{sec:proof} we prove Theorem \ref{thm:BSofC2b}.  Finally, in section \ref{ch:generalize} we give some conjectures based on these results.

%%%%%%%%%%%%%%%%%%%%%%%%%%%%%%%%%%%%%%%%%%%%%%%%%%%%%%%%%%%%%%%%%%%%%%%%%%%%
%%%%%%%%%%%%%%%%%%%%%%%%%%%%%%%%%%%%%%%%%%%%%%%%%%%%%%%%%%%%%%%%%%%%%%%%%%%%
%%%%%%%%%%%%%%%%%%%%%%%%%%%%%%%%%%%%%%%%%%%%%%%%%%%%%%%%%%%%%%%%%%%%%%%%%%%%

\section{Notation and Background} \label{sec:notation}

%%%%%%%%%%%%%%%%%%%%%%%%%%%%%%%%%%%%%%%%%%%%%%%%%%%%%%%%%%%%%%%%%%%%%%%%%%%%
%%%%%%%%%%%%%%%%%%%%%%%%%%%%%%%%%%%%%%%%%%%%%%%%%%%%%%%%%%%%%%%%%%%%%%%%%%%%

\subsection{Preliminaries}\label{subsec:preliminaries}

We assume a general knowledge of the basics of knot theory and 3-manifolds.  See \cite{Hatcher} and \cite{Rolfsen} for more detail on any topics in this subsection. A {\em knot} $K$ is an embedded copy of $S^1$ in $S^3$. More precisely, a knot $K$ is the isotopy class of images of embeddings of $S^1\hookrightarrow S^3$. A {\em link} is multiple copies of $S^1$ simultaneously embedded in $S^3$.  Two knots are {\em equivalent} if there is an ambient isotopy taking one knot to the other. Let $N(\cdot )$ and $\eta(\cdot)$ denote closed and open regular neighborhoods.  The space $E(K):=S^3-\eta(K)$ is called the {\em exterior}\label{def:exterior} of the knot $K$ and $\p E(K)$ is a torus.  We will use $|A|$ to denote the number of connected components of $A$.

Let $M$ be a 3-manifold, $S \subset M$ an embedded surface and $J\subset M$ an embedded 1-manifold. We will use $M(J)$ and $S_J$ to denote $M-\eta(J)$ and $S - \eta(J)$, respectively. We use the notation  $(M,J)$ to denote the pair of $M$ and $J$.   An embedded 1-manifold in a surface $\gamma \subset S$ is {\em essential} if $S\setminus \eta(\gamma)$ does not have a disk component in $S$.  A submanifold $M'\subset M$ is {\em proper}\label{def:proper} if $\p M' \subset \p M$. A {\em compressing disk}\label{def:compressing} $D$ for $S$ is a embedded disk in $M$ with $ D\cap S=\partial D$ but $\partial D$ does not bound a disk in $S$.  A {\em boundary compressing disk}, or {\em $\p$-compressing disk}, $\Delta$ for $S$ is an embedded disk in $M$ such that $\Delta \cap S = \gamma$  is an arc with $\gamma \subset \p \Delta$, $\gamma$ is essential in $S$ and $\p \Delta - \gamma $, another arc, is a subset of $\p M$.  The surface $S$ is  {\em incompressible} if there does not exist a compressing disk $D$ for $S$; $S$ is {\em $\p$-incompressible} if there does not exist a $\p$-compressing disk $\Delta$ for $S$.  The surface $S$ is said to be {\em essential} if $S$ is incompressible, $\p$-incompressible and not isotopic into $\p M$.

%%%%%%%%%%%%

\subsection{2-bridge knots and rational tangles}\label{subsec:2bridge}

One way to create a 2-bridge knot is to start with a rational tangle.

\begin{definition} \label{2-bridge knot}
Given a rational number $p/q \in \Q$, write
\[
\frac{p}{q} =  r+\cfrac{1}{b_1-\cfrac{1}{b_2-\cfrac{1}{\ddots -\cfrac{1}{b_k } }}} 
\]  such that $r\in \Z$, $b_i \in \Z \setminus \{0\}$, then $r+[b_1,b_2, \ldots, b_k]$ denotes a {\em partial fraction decomposition} of $p/q$.

\end{definition}

It should be noted that partial fraction decompositions are not unique in general.  But if we assume the $b_i$ all have the same sign and $k$ is odd, then the partial fraction decomposition is unique, see \cite{KL}. Given a partial fraction decomposition $\frac{p}{q}= r+[b_1,b_2, \ldots, b_k]$, one can produce a diagram that is formed from a vertical 3 braid with $a_1$ half-twists between the left two strands, below that, add $a_2$ half-twists between the right two strands.  Then alternate between the left and right two strands in this way for each consecutive $a_i$. Finally, add a fourth strand on the left, cap and connect the new left two strands at the top and connect the right two strands at the bottom. The diagram we have created is called a {\em 4-plat rational tangle}\label{4platTangle}.  See Figure~\ref{fig:2bridgeTangle}.

\begin{figure}[htbp]
\begin{center}

\begin{tikzpicture}[scale = .7]

%2 bridge representation of a rational tangle

%\draw[help lines] (0,0) grid (15,10);

\draw (0.5,4.625) rectangle +(2,2);

\draw (1,4.625) -- +(0,-1);
\draw (2,4.625) -- +(0,-1);
\draw (1,6.625) -- +(0,1);
\draw (2,6.625) -- +(0,1);

\node at (1.5,5.625) {$\alpha,\beta$};

\draw [thick] (3.5,5.5) -- (4,5.5);
\draw [thick] (3.5,5.75) -- (4,5.75);

\draw [dashed] (5,2) rectangle +(5.5,7.25);

\draw(5.5,1.25) -- (5.5, 9.5);

\draw(7,9) -- (7, 9.5);
\draw(7,5.25) -- (7, 7.5);
\draw(7,4) -- (7, 4.25);
\draw(7,1.25) -- (7, 2.5);

\draw(8.5,1.75) -- (8.5, 2.5);
\draw(8.5,4) -- (8.5, 4.25);
\draw(8.5,5.25) -- (8.5, 5.5);
\draw(8.5,7) -- (8.5, 7.5);
\draw(8.5,9) -- (8.5, 10);

\draw(10,1.75) -- (10, 4.25);
\draw(10,5.25) -- (10, 5.5);
\draw(10,7) -- (10, 10);

\draw [domain=0:180] plot ({6.25+.75*cos(\x)}, {9.5+.75*sin(\x)});	
\draw [domain=180:360] plot ({9.25+.75*cos(\x)}, {1.75+.75*sin(\x)});

%a_1

\draw [domain=0:.4] plot ({.75*cos(pi* \x r )+7.75}, .5*\x+7.5);

\draw [domain=0:1.4] plot ({-.75*cos(pi* \x r )+7.75}, .5*\x+7.5);

\draw [domain=.6:2.4] plot ({.75*cos(pi* \x r )+7.75}, .5*\x+7.5);

\draw [domain=.6:2] plot ({.75*cos(pi* \x r )+7.75}, .5*\x+8);
\draw [domain=.6:1] plot ({.75*cos(pi* \x r )+7.75}, .5*\x+8.5);

%a_2

\draw [domain=0:.4] plot ({.75*cos(pi* \x r )+9.25}, .5*\x+5.5);

\draw [domain=0:1.4] plot ({-.75*cos(pi* \x r )+9.25}, .5*\x+5.5);

\draw [domain=.6:2.4] plot ({.75*cos(pi* \x r )+9.25}, .5*\x+5.5);

\draw [domain=.6:2] plot ({.75*cos(pi* \x r )+9.25}, .5*\x+6);
\draw [domain=.6:1] plot ({.75*cos(pi* \x r )+9.25}, .5*\x+6.5);

%a_m
\draw [domain=0:.4] plot ({.75*cos(pi* \x r )+7.75}, .5*\x+2.5);

\draw [domain=0:1.4] plot ({-.75*cos(pi* \x r )+7.75}, .5*\x+2.5);

\draw [domain=.6:2.4] plot ({.75*cos(pi* \x r )+7.75}, .5*\x+2.5);

\draw [domain=.6:2] plot ({.75*cos(pi* \x r )+7.75}, .5*\x+3);
\draw [domain=.6:1] plot ({.75*cos(pi* \x r )+7.75}, .5*\x+3.5);

\draw [fill] (8.5,4.5) circle [radius=0.04];	
\draw [fill] (8.5,4.75) circle [radius=0.04];	
\draw [fill] (8.5,5) circle [radius=0.04];

\node at (6.5,3.25) {$a_{m}$};
\node at (8,6.25) {$a_{2}$};
\node at (6.5,8.25) {$a_{1}$};

\end{tikzpicture}

\caption{A rational tangle in the 4-plat form.}\label{fig:2bridgeTangle}
\end{center}

\end{figure}
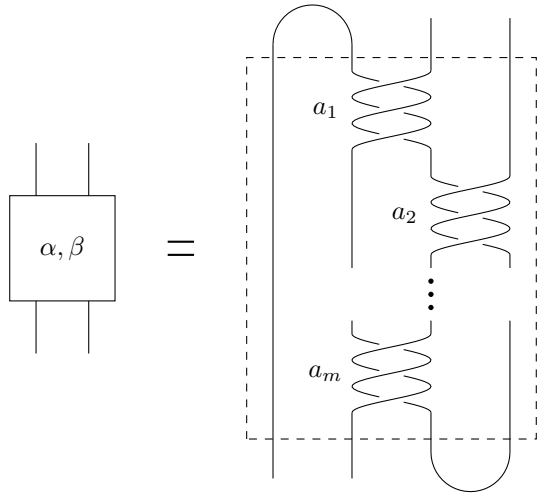

The ``pillowcase" is a term used by Hatcher and Thurston in \cite{HT}.  Technically, one should define it as $I^2\sqcup_{\p I^2}I^2$, where one copy of $I^2$ is the ``front" of the pillowcase and the other is the ``back."  A {\em 2-bridge tangle}\label{2bTangle} is given by a fraction $p/q\in \Q$ by placing lines of slope $p/q$ on the front of the pillowcase and connecting them with lines of slope $-p/q$ on the back of the pillowcase.  See Figure \ref{fig:rationaltangle}. For a more detailed treatment of this subsection, see \cite{HT}.

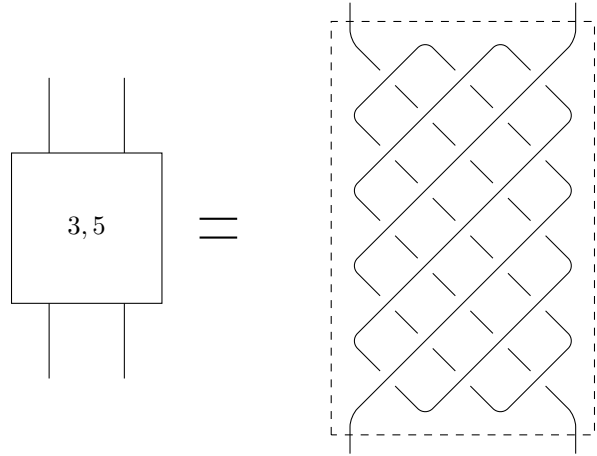
\begin{figure}[htbp]
\begin{center}

\begin{tikzpicture}

%rational tangle

%\draw[help lines] (0,0) grid (15,10);

\draw (-4.5,1.5) rectangle +(2,2);

\draw (-4,1.5) -- +(0,-1);
\draw (-3,1.5) -- +(0,-1);
\draw (-4,3.5) -- +(0,1);
\draw (-3,3.5) -- +(0,1);

\node at (-3.5,2.5) {$3,5$};

\draw [thick] (-2,2.625) -- +(.5,0);
\draw [thick] (-2,2.375) -- +(.5,0);

\draw [dashed] (-.25,-.25) rectangle +(3.5,5.5);

%tangle
%positive sloped lines
\draw [rounded corners] (.4,3.6) -- (0,4) -- (1,5) -- (1.4,4.6);
\draw [rounded corners] (.4,2.6) -- (0,3) -- (2,5) -- (2.4,4.6);
\draw [rounded corners] (0,-.5) -- (0,0) -- (3,3) -- (2.6,3.4);
\draw [rounded corners] (.4,.6) -- (0,1) -- (3,4) -- (2.6,4.4);
\draw [rounded corners] (.4,1.6) -- (0,2) -- (3,5) -- (3,5.5);
\draw [rounded corners] (.6,.4) -- (1,0) -- (3,2) -- (2.6,2.4);
\draw [rounded corners] (1.6,.4) -- (2,0) -- (3,1) -- (2.6,1.4);

%negative sloped lines
\draw (.6, 1.4) -- +(.3,-.3);
\draw (1.1, .9) -- +(.3,-.3);

\draw (.6, 2.4) -- +(.3,-.3);
\draw (1.1, 1.9) -- +(.3,-.3);
\draw (1.6, 1.4) -- +(.3,-.3);
\draw (2.1, .9) -- +(.3,-.3);
\draw [rounded corners] (2.6,.4) -- (3,0) -- (3,-.5) ;

\draw (.6, 3.4) -- +(.3,-.3);
\draw (1.1, 2.9) -- +(.3,-.3);
\draw (1.6, 2.4) -- +(.3,-.3);
\draw (2.1, 1.9) -- +(.3,-.3);

\draw (.6, 4.4) -- +(.3,-.3);
\draw (1.1, 3.9) -- +(.3,-.3);
\draw (1.6, 3.4) -- +(.3,-.3);
\draw (2.1, 2.9) -- +(.3,-.3);
\draw [rounded corners] (.4,4.6) -- (0,5) -- (0,5.5);

\draw (1.6, 4.4) -- +(.3,-.3);
\draw (2.1, 3.9) -- +(.3,-.3);

\end{tikzpicture}

\caption{A 2-bridge rational tangle of $\frac{3}{5}$.}\label{fig:rationaltangle}
\end{center}

\end{figure}

%%%%%%%%%%%%%%%%%%%%%%%%%%%%%

These two different diagrams yield isotopic tangles if and only if they represent the same rational number.  We can create a {\em 2-bridge knot} by taking a 4-plat rational tangle or the 2-bridge tangle and connecting the top two strands together and the bottom two strands together.

%%%%%%%%%%%%%%%%%%%%%%%%%%%%%%%%%%%%%%%%%%%%%%%%%%

% A 2-bridge knot

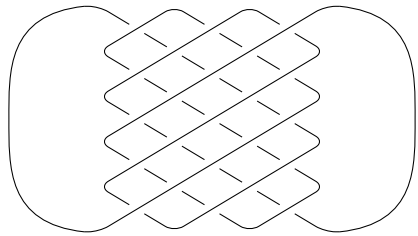
\begin{figure}
\begin{center}
\begin{tikzpicture}[yscale=.6]

%rational tangle

%\draw[help lines] (-1,-1) grid (4,7);

%tangle
%positive sloped lines
\draw [rounded corners] (.4,3.6) -- (0,4) -- (1,5) -- (1.4,4.6);
\draw [rounded corners] (.4,2.6) -- (0,3) -- (2,5) -- (2.4,4.6);
\draw [rounded corners] (-.2,0) -- (0,0) -- (3,3) -- (2.6,3.4);
\draw [rounded corners] (.4,.6) -- (0,1) -- (3,4) -- (2.6,4.4);
\draw [rounded corners] (.4,1.6) -- (0,2) -- (3,5) -- (3.2,5);
\draw [rounded corners] (.6,.4) -- (1,0) -- (3,2) -- (2.6,2.4);
\draw [rounded corners] (1.6,.4) -- (2,0) -- (3,1) -- (2.6,1.4);

%negative sloped lines
\draw (.6, 1.4) -- +(.3,-.3);
\draw (1.1, .9) -- +(.3,-.3);

\draw (.6, 2.4) -- +(.3,-.3);
\draw (1.1, 1.9) -- +(.3,-.3);
\draw (1.6, 1.4) -- +(.3,-.3);
\draw (2.1, .9) -- +(.3,-.3);
\draw [rounded corners] (2.6,.4) -- (3,0) -- (3.2,0) ;
\draw[] (3.2,0) to  [out=10,in=270] (4.2,2.5) to  [out=90,in=-10] (3.2,5) ;

\draw (.6, 3.4) -- +(.3,-.3);
\draw (1.1, 2.9) -- +(.3,-.3);
\draw (1.6, 2.4) -- +(.3,-.3);
\draw (2.1, 1.9) -- +(.3,-.3);

\draw (.6, 4.4) -- +(.3,-.3);
\draw (1.1, 3.9) -- +(.3,-.3);
\draw (1.6, 3.4) -- +(.3,-.3);
\draw (2.1, 2.9) -- +(.3,-.3);
\draw [rounded corners] (.4,4.6) -- (0,5) -- (-.2,5);
\draw[] (-.2,0) to  [out=170,in=270] (-1.2,2.5) to [out=90,in=190] (-.2,5) ;

\draw (1.6, 4.4) -- +(.3,-.3);
\draw (2.1, 3.9) -- +(.3,-.3);

\end{tikzpicture}

\caption{The 2-bridge knot corresponding to $\frac{3}{5}$.}\label{fig:2-bridge}
\end{center}
\end{figure}

%%%%%%%%%%%%%%%%%%%%%%%%%%%%%%%%%%%%%%%%%%%%%%%%%%

%%%%%%%%%%%%%%%%%%

\subsection{Bridge Spectrum}\label{sec:bridgespectrum}

Given a 3-manifold $M$ with boundary, a {\em trivial arc} is a properly embedded arc $\alpha$, see page \pageref{def:proper}, that cobounds a disk with an arc $\beta \subset \partial M$; i.e., $\alpha \cap \beta =\partial \alpha =\partial \beta$ and there is an embedded disk $D \subset M$ such that $\partial D = \alpha \cup \beta$. We call the disk $D$ a {\em bridge disk}.  A {\em bridge splitting} of a knot $K$ in $M$ is a decomposition of $(M,K)$ into $(V_1,A_1) \cup_\Sigma (V_2,A_2)$, where each $V_i$ is a handlebody with boundary $\Sigma$ and $A_i \subset V_i$ is a collection of trivial arcs for $i=1,2$.  One should note that when we exclude the knot $K$, a bridge splitting is a {\em Heegaard splitting of $M$}.   For all $g$ and $b$, a {\em $(g,b)$-splitting} of $(M,K)$ is a bridge splitting with $g(V_i)=g$ and $|A_i|=b$ for $i=1,2$.  The surface $\Sigma$ is called a {\em bridge surface}\label{def:bridgesurface}.  Throughout this paper, we will be focusing on bridge splittings of $(S^3,K)$.

\begin{definition} \label{def:genus $g$ bridge number}

The {\em genus $g$ bridge number} of a knot $K$ in $S^3$, $b_g(K)$, is the minimum $b$ such that a $(g,b)$-splitting of $K$ exists. We also require that for $b_g (K)$ to be zero, the knot must be able to be isotoped into a genus $g$ Heegaard surface.

\end{definition}

\begin{definition} \label{def:bridge spectrum}

The {\em bridge spectrum} of a knot $K$ in $S^3$, ${\mathbf b}(K)$, is the list of genus $g$ bridge numbers:

$$(b_0(K),b_1(K),b_2(K),\ldots ).$$

\end{definition}

As mentioned earlier, the genus zero bridge number is the classical bridge number, except in the case of the unknot. A simple closed curve in the boundary of a handlebody is called {\em primitive} if it transversely intersects the boundary of a properly embedded essential disk of the handlebody in a single point.  Some define $b_g(K)=0$ when $K$ can be isotoped into a genus $g$ Heegaard surface and is primitive.  Here is another invariant, $\hat{b}_g(K)$ with the following added requirement. 

\begin{definition} \label{def:primitive genus $g$ bridge number}

The {\em (primitive) genus $g$ bridge number} of a knot $K$ in $S^3$, $\hat{b}_g(K)$, is the minimum $b$ such that a $(g,b)$-splitting of $K$ exists. We also require that for $b_g (K)$ to be zero, the knot must be able to be isotoped into a genus $g$ Heegaard surface and is primitive.

\end{definition}

\begin{definition} \label{def:primitive bridge spectrum }

The {\em (primitive) bridge spectrum} of a knot $K$ in $S^3$, $ {\mathbf {\hat b}}(K)$ is the list of genus $g$ bridge numbers:

$$(\hat{b}_0(K),\hat{b}_1(K),\hat{b}_2(K),\ldots ).$$

\end{definition}

In the next subsection we will see that the process of meridional stabilization forces the bridge spectrum and primitive bridge spectrum to be strictly decreasing, see Proposition \ref{prop:BSdecrease}.  Thus, when $b_g(K)=0$, we can discard the bridge number for higher genus than $g$. The only potential difference between these two spectra is the last non-zero term in the sequence.  We make the relation between bridge spectrum and primitive bridge spectrjm rigorous for future use.

\begin{proposition}\label{prop:bsVSpbs}
For a knot $K$, if $b_g(K)\not =0$, then $b_g(K)=\hat{b}_g(K)$.  If $b_g(K)=0$, then $\hat{b}_g(K)\in \{ 0,1\}.$
\end{proposition}

\begin{proof}
If $ b_g(K)\geq1$, then by definition, $\hat{b}_g(K)=b_g(K)$. If $b_g(K)=0$, then $K$ embeds in a genus $g$ Heegaard surface, $\Sigma$.  If $K$ is primitive on one side of $\Sigma$, then $\hat{b}_g(K)=0$.  If $K$ is not primitive on either side of $\Sigma$, then through the process of elementary stabilization, we can create a $(g,1)$-splitting.  Hence, $\hat{b}_g(K)\leq 1$, completing the proof.  \end{proof}

For example, consider the following well known proposition:

\begin{proposition}\label{prop:BSofT}
Given a non-trivial torus knot $K=T_{p,q}$ in $S^3$, we have ${\mathbf b}(K)=(\min\{p,q\},0)$, while ${\mathbf {\hat b}}(K)=(\min\{p,q\},1,0)$.
\end{proposition}

\begin{proof}
Schubert proved that $b_0(K)={\hat b}_0(K)=\min\{p,q\}$, which is greater than one for nontrivial knots. For $b_1(K)$, by definition, a torus knot embeds on a genus one surface.  Hence, $b_1(K)=0$.  For ${\hat b}_1(K)$, we see that a nontrivial torus knot cannot embed on a torus and intersect an essential disk of a genus one torus only once.  By definition of a torus knot, it must intersect the meridian disk $p$ times and the longitudinal disk $q$ times. Hence, ${\hat b}_1(K)\geq 1$, but we can easily see that we can make a $(1,1)$-splitting, thus ${\hat b}_1(K)=1$. And for ${\hat b}_2(K)$, we can easily embed a torus knot on a genus two surface with one handle having only a single arc of the knot.\end{proof}

So they are distinct invariants. But there are knots for which they coincide. The proof above also gives us part of the proof of the following proposition.  The other direction follows directly from the definition of genus one bridge number.

\begin{proposition}
A knot $K$, in $S^3$, is a torus knot if and only if $b_1(K)=0$.
\end{proposition}

Thus, any non-torus knot has $b_1(K)\geq 1$, which, along with the fact that bridge spectrum are strictly decreasing sequences, gives us the following proposition.

\begin{proposition}\label{prop:BSof2b}
Let $K$ be a 2-bridge knot that is not a torus knot in $S^3$; then ${\mathbf {\hat b}}(K)={\mathbf b}(K)=(2,1,0)$.  
\end{proposition}

\subsection{Operations on bridge splittings and multiple bridge splittings}\label{sec:Operations}

This subsection is a summary of tools that we will need in the proof of our theorem.  Many come from work of numerous people but directly these can also be found in \cite{Z1}. There are three main ways to operate on a bridge surface of a knot  in $S^3$ to obtain a new bridge splitting: stabilization, perturbation, and meridional stabilization.

The genus of a surface can be increased by adding a handle which does not interact with the knot through a process called {\em elementary stabilization}.  A properly embedded arc $\alpha$, see page \pageref{def:proper}, is said to be {\em boundary parallel} in a 3-manifold $M$ if it isotopic rel boundary into $\p M$.  Let $(S^3,K)=(V_1,A_1) \cup_\Sigma (V_2,A_2)$ and let $\alpha$ be a boundary parallel arc in $V_1$ such that $\alpha\cap A_1 = \emptyset$. Then let $W_1=V_1 - \eta(\alpha)$, $W_2=V_2\cup N(\alpha)$, and $\Sigma'=\p W_1=\p W_2$.  Then  $(S^3,K)=(W_1,A_1) \cup_{\Sigma'} (W_2,A_2)$ is a new bridge splitting with the genus of $\Sigma'$ one higher than $\Sigma$.  We can also run this process in reverse.  If $D_i$ are compressing disks in $(V_i,A_i)$ for $\Sigma - \eta(K)$, and $\vert D_1\cap D_2\vert =1$, then $\p N(D_1\cup D_2)$ is a 2-sphere which intersects $\Sigma$ in a single curve.  Then compression along this curve yields a new bridge surface $\Sigma''$ of lower genus, and $\Sigma$ is said to be {\em stabilized}.

The number of trivial arcs in $A_1$ and $A_2$ can be increased by one each, through the process of {\em elementary perturbation}.  Add in a canceling pair of trivial arcs in a sideways ``S" shape, cutting through the surface; i.e. a strand which passed down through the surface transversally can be perturbed and a subsection of the arc can be brought up through the surface again, creating two new arcs. Again, for the reverse direction, if there are two bridge disks on either side of $\Sigma$, which intersect in a single point contained in $J$, one may construct an isotopy which cancels two arcs of $A_1$ and $A_2$, creating a new surface $\Sigma''$, and $\Sigma$ is {\em perturbed}. 

Bridge spectra are always strictly decreasing sequences.  To see this, take any $(g,b)$-splitting of a knot.  Then, by definition, we have a decomposition of $(S^3,k)$ as  $(V_1,A_1) \cup_\Sigma (V_2,A_2)$.  Intuitively, take any trivial arc $\alpha$ in, say, $A_1$; then $N(\alpha)\subset V_1$ is a closed neighborhood of $\alpha$ in $V_1$. Take this neighborhood from $V_1$ and move it to $V_2$, which produces one higher genus handlebodies, and one less trivial arc. More carefully, let $W_1=(V_1-\eta(\alpha))$ and $W_2=(V_2 \cup N(\alpha))$.  Let $B_1 = A_1 - \{\alpha\}$, and $B_2 = A_2 \cup \alpha$. Then $B_2$ is $A_2$ with two arcs combined into a single arc by connecting them with $\alpha$.  Also, notice that $g(W_i)=g(V_i)+1$.   Thus, we have  $(S^3,k)=(W_1,B_1) \cup_{\Sigma'} (W_2,B_2)$, which is a $(g+1,b-1)$-splitting. This process is called {\em meridional stabilization}\label{def:MeridionalStabilize}. This process proves the following proposition.

\begin{proposition}\label{prop:BSdecrease}
If $K$ is a knot in $S^3$, with $\hat{b}_g(K)\geq1$, then $\hat{b}_{g+1}(K) \leq \hat{b}_g(K)-1$.  Similarly, if ${b}_g(K)\geq1$, then ${b}_{g+1}(K) \leq {b}_g(K)-1$.  
\end{proposition}

The next corollary is immediate from Proposition \ref{prop:BSdecrease}.

\begin{corollary}\label{cor:boundedabove}
If $K$ is a knot in $S^3$, then its bridge spectrum and primitive bridge spectrum are bounded above, component-wise, by $( b_0(K), b_0(K)-1, b_0(K)-2, \ldots )$. Equivalently, for every knot $K$, and every $g\leq b_0(K)$, there is a $(g,b_0(K)-g)$-splitting for $K$.
\end{corollary}

For a bridge surface $\Sigma$ in $(S^3,K)$, one can sometimes find two bridge disks on opposite sides of $\Sigma$ which have two points of intersection in $K$.  In this case, the component of $J$ is isotopic into $\Sigma$ and $\Sigma$ is called {\em cancelable.}

A $(g,b)$-surface for a knot $K$ is said to be {\em irreducible} if it is not stabilized, perturbed, meridionally stabilized, or cancelable.  Thus, if $b_g(K)<b_{g-1}(K)-1$, then a $(g,b)$-surface $\Sigma$ satisfying $b=b_g(K)$ must be irreducible. 

In this paragraph, we will use definitions on page \pageref{def:compressing}. Let $\Sigma$ be a bridge surface for $(S^3,K)$ which yields the splitting $(S^3,K) =(V_1,A_1) \cup_\Sigma (V_1,A_1)$. Then $\Sigma $ is called {\em weakly reducible} if there exist disjoint disks $D_1$ and $D_2$, that are either both compressing, both bridge, or one of each, such that $D_i \subset (V_i,A_i)$ for $i=1,2$, for $\Sigma_J$. If $\Sigma$ is not weakly reducible, perturbed, or cancelable, then $\Sigma$ is called {\em strongly irreducible.} By considering bridge disks as embedded in $M(J)$, one can see that perturbed and cancelable surfaces will be weakly reducible; hence, in $M(J)$, $\Sigma$ is strongly irreducible if and only if it is not weakly reducible.

For a more details about the following concepts, see \cite{HS}.  Let $F$ be a disjoint union of closed oriented surfaces. A {\em compression body} $C$ is a handlebody or the 3-manifold obtained by attaching 1-handles to $F\times \{1\}\subset F\times I$.  Then let $\p_-C=F\times \{0\}$ and let $\p_+C=\p C - \p_-C$.  An arc in a compression body is said to be {\em vertical} if it is isotopic to $\{x\}\times I$ for $x\in F$. Next, a {\em multiple bridge splitting} is the following: let $(M,J)$ contain a collection $\mathscr{S}=\{ \Sigma_0,S_1,\Sigma_1,\ldots,S_d,\Sigma_d\}$ of disjoint surfaces transverse to $J$, such that $(M,J)$ cut along $\mathscr{S}$ is a collection of compression bodies containing trivial arcs $\{ (C_0,\tau_0),(C_0',\tau_0'),\ldots,(C_d,\tau_d),(C_d',\tau_d')\}$, where

\begin{itemize}
\item $(C_i,\tau_i)\cup_{\Sigma_i}(C_i',\tau_i')$ is a bridge splitting of a submanifold $(M_i,J_i)$, where $M_i = C_i \cup C_i'$ and $J_i = \tau_i \cup \tau_i'$, for each $i$,
\item $\p_-C_i=\p_-C_{i-1}'=S_i$ for $1\leq i \leq d$,
\item $\p M = \p_-C_0 \cup \p_-C_d'$, and
\item $J=\bigcup_{i=1}^d(\tau_i \cup \tau_i')$.
\end{itemize}

The surfaces $\Sigma_i$ are called {\em thick} and the surfaces $S_j$ are called {\em thin}.  The thick surface $\Sigma_i$ is {\em strongly irreducible} if it is strongly irreducible in the manifold $(C_i,\tau_i)\cup_{\Sigma_i}(C_i',\tau_i')$, and a multiple bridge splitting is called {\em strongly irreducible} if each thick surface is strongly irreducible and no compression body is trivial. A compression body is trivial if it is homeomorphic to $\Sigma_i \times I$ with $\tau_i$ only vertical arcs.   This leads us to the following theorem of Hayashi and Shimokawa \cite{HS}, which we present in the same way that Zupan does in \cite{Z1}. Theorem \ref{thm:HS} is the basis for one of the two major cases in the proof of Theorem \ref{thm:BSofC2b}.

\begin{theorem}\label{thm:HS}{{\cite{HS},\cite[Theorem 2.8]{Z1}}}
Let $M$ be a 3-manifold containing a 1-manifold J. If $(M,J)$ has a strongly irreducible multiple bridge splitting, then $\p (M-\eta(J))$ and every thin surface are incompressible.  On the other hand, if $\p (M-\eta(J))$ is incompressible in $ M-\eta(J)$ and $\Sigma$ is a weakly reducible bridge splitting for $(M,J)$, then $(M,J)$ has a strongly irreducible multiple bridge splitting $\{ \Sigma_0,S_1,\Sigma_1,\ldots,S_d,\Sigma_d\}$  satisfying
$$ g(\Sigma)= \sum_{i=0}^d g(\Sigma_i) -\sum_{i=1}^d g(S_i).$$
\end{theorem}

\subsection{Cable spaces and Cables}\label{sec:cable}

A cable space, loosely, is the solid torus with a torus knot taken out of its interior.  For more information on cable spaces, see \cite{GL} and \cite{Z1}. More precisely, let $T_{m,n}$ be the torus knot on the standardly embedded torus in $S^3$ intersecting the meridian transversely in $m$ points and intersecting the longitude transversely in $n$ points. Let $V$ be the solid torus $S^1\times D^2$ with the torus knot pushed into the interior of $V$ off the boundary.  The {\em cable space}\label{def:cablespace} $C_{m,n}$ is $V-\eta(T_{m,n})$. Cable spaces are Seifert fibered. For more information on Seifert fibered spaces, see Waldhausen \cite{Wald}.  There one will find that essential surfaces, see page \pageref{def:proper}, in Seifert fibered spaces are either vertical or horizontal. A {\em vertical} surface in a Seifert fibered space made up of a union of fibers. A {\em horizontal} surface in a Seifert fibered space is transverse to any fiber it intersects. In a cable space $C_{p,q}$, the notation $\p_+C_{p,q}$ is used to denote the boundary of the solid torus $V$ and $\p_-C_{p,q}$ denotes the boundary of the $(p,q)$-torus knot.  

The following are Zupan's Lemmas 3.2 and 3.3.

\begin{lemma}\label{lem:Z3.2}{{\cite[Lemma 3.2]{Z1}}}
Suppose $S\subset C_{p,q}$ is incompressible.  If each component of $S\cap \p_+C_{p,q}$ has integral slope, then each component of $S\cap \p_-C_{p,q}$  also has integral slope.\end{lemma}

\begin{lemma}\label{lem:Z3.3}{{\cite[Lemma 3.3]{Z1}}}
Suppose $S\subset C_{p,q}$ is incompressible.  Then $S\cap \partial_+C_{p,q}$ is meridional if and only if  $S\cap \partial_-C_{p,q}$ is meridional.
\end{lemma}

A {\em cable} of a knot is defined in the following way.  Given a knot $K_0$ in $S^3$, and a torus knot $T_{m,n}$, the cable $K$ we denote by ${\mathsf {cable}}(T_{m,n}, K_0)$ is the knot obtained by taking $S^3\setminus \eta(K_0)$ and gluing in the solid torus $V$ with $T_{m,n}$ pushed slightly into the interior of $V$.  The space $V$ is glued in so that the cable has the preferred framing, that is, the usual longitude of $V$ is mapped to the the trivial element in $H_1(E(K_0))$.  Then the knot $K$ is called the {\em $(m,n)$-cable} of $K_0$.

%\coment{Margin notes to yourself and coauthors of what still needs to be written can be put here.}

%%%%%%%%%%%%%%%%%%%%%%%%%%%%%%%%%%%%%%%%%%%%%%%%%%%%%%%%%%%%%%%%%%%%%%%%%%%%
%%%%%%%%%%%%%%%%%%%%%%%%%%%%%%%%%%%%%%%%%%%%%%%%%%%%%%%%%%%%%%%%%%%%%%%%%%%%
%%%%%%%%%%%%%%%%%%%%%%%%%%%%%%%%%%%%%%%%%%%%%%%%%%%%%%%%%%%%%%%%%%%%%%%%%%%%

%%%%%%%%%%%%%%%%%%%%%%%%%%%%%%%%%%%%%%%%%%%%%%%%%%

%%%%%%%%%%%%%%%%%%%%%%%%%%%%%%%%%%%%%%%%%%%%%%%%%%

%%%%%

\subsection{Incompressible surfaces in 2-bridge knot complements}\label{sec:HT}
This subsection will mostly be devoted to Hatcher's and Thurston's result about essential surfaces, in the complement of 2-bridge knots. See \cite{HT} for more information. 

We will devote this paragraph to describing how these surfaces $S_n(n_1,n_2,\ldots,n_{k-1})$ are defined. Given a continued fraction decomposition $r+[b_1,b_2, \ldots, b_k]$ of $\frac{p}{q}$, form the corresponding 4-plat rational tangle in $S^3$, see Definition \ref{4platTangle}, and create a link by connecting the top two strands together and bottom two strands together. Since we are only investigating knots in this paper, we will assume that $q$ is odd in the reduced fraction $\frac{p}{q}$.  Isotope the knot into a vertical square tower, see Figure \ref{fig:2bridgeTower|}.  This knot has $k-1$ {\em inner horizontal plumbing squares}, one between each twisted region.  For each inner plumbing square, there is a complementary {\em outer plumbing square} which exists in the same horizontal plane in $S^3$, thinking of $S^3$ as $(S^2\times [0,1])/(S^2\times\{0\},S^2\times\{1\})$, see Figure \ref{fig:plumbSquare}.  The surface $S_n(n_1,n_2,\ldots,n_{k-1})$, where $n\geq 1$ and $0\leq n_i\leq n$, consists of $n$ parallel sheets running close to the vertical bands of this tower form of the knot.  At the $i$-th plumbing square, $n_i$ of the $n$ sheets run into the inner plumbing square and the other $n-n_i$ sheets run into the outer plumbing square.  See Figure \ref{fig:figure8Asurface} for a small example.

The branched surface $\Sigma [b_1,\ldots,b_k]$ is obtained by a single sheet running vertically and branching at each plumbing square into the inner and the outer plumbing square, see Figure \ref{fig:plumbSquare}. So for each continued fraction decomposition $r+[b_1,b_2, \ldots, b_k]$, the branched surface $\Sigma [b_1,\ldots,b_k]$ carries many not necessarily connected surfaces $S_n(n_1,n_2,\ldots,n_{k-1})$.  

Hatcher and Thurston, \cite{HT}, give the following classification of incompressible surfaces. 

%%%%%%%%%%%%

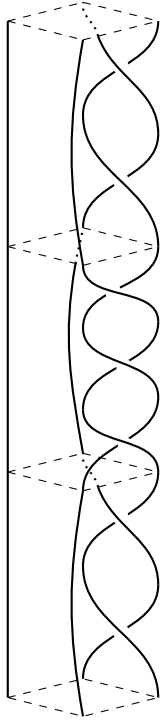
\begin{figure}
\begin{center}

\begin{tikzpicture}[yscale=1]

%\draw[help lines] (0,0) grid (4,7);

\draw [thick] (0,-3) to (0,6);
%squares
\draw [dashed] (0,6) to (1,6.25);
\draw [dashed] (0,6) to (1,5.75);
\draw [dashed] (1,5.75) to (2,6);
\draw [dashed] (1,6.25) to (2,6);

\draw [dashed] (0,3) to (1,3.25);
\draw [dashed] (0,3) to (1,2.75);
\draw [dashed] (1,2.75) to (2,3);
\draw [dashed] (1,3.25) to (2,3);

\draw [dashed] (0,0) to (1,.25);
\draw [dashed] (0,0) to (1,-.25);
\draw [dashed] (1,-.25) to (2,0);
\draw [dashed] (1,.25) to (2,0);

\draw [dashed] (0,-3) to (1,-2.75) to (2,-3) to (1,-3.25) to (0,-3);
%top twist region
%\draw [thick] (0,6) to [out=90, in=90] (1,6.25);

\draw [thick] (.9,2.77) to  [out=260,in=100] (1,.25);
\draw [thick, dotted] (.9,2.77) to [out=80, in=260] (1,3.25);

\draw [thick, dotted] (1.2,5.8) to [out=100, in=260] (1,6.25);
\draw [thick, dotted] (1.2,-0.2) to [out=100, in=260] (1,.25);

%\draw [fill=white, white] (.98,3) circle [radius=0.1];

\draw [thick] (1,5.75) to  [out=260,in=100] (1,2.75);

\draw [thick] (1,5.75-6) to  [out=260,in=100] (1,2.75-6);

%\draw [very thick] (1.4,5.45-1.5) to  (1.55, 5.3-1.5); 
\draw[thick] (2,6) to [out=270,in=90] (1,4.75);
\draw [fill=white, white] (1.5,5.4) circle [radius=0.1];

\draw[thick] (1.2,5.8) to [out=290,in=90] (2,4.5);
%\draw [fill=white, white] (1.05,5.95) circle [radius=0.1];

%\draw [thick] (1,5.75) to [out=90, in=90] (2,6);

%\draw[very thick] (1,6.25) to [out=270,in=90] (2,4.5) to [out=270, in=90] (1,3.25);

\draw[thick] (2,4.5) to [out=270, in=90] (1,3.25);
\draw [fill=white, white] (1.5,3.88) circle [radius=0.1];
\draw[thick] (1,4.75) to [out=270, in=90] (2,3);

%middle twist region

\draw [thick] (2,3) to [out=270, in=90] (1,1.9166);
\draw [fill=white, white] (1.4,2.4) circle [radius=0.1];

\draw [thick] (1,2.75) to [out=270, in=90] (2,2);

\draw[thick] (2,2) to [out=270, in=90] (1,.8333);
\draw [fill=white, white] (1.55,1.4) circle [radius=0.1];

\draw [thick] (1,1.9166) to [out=270, in=90] (2,1);

\draw [thick] (2,1) to [out=270, in=90] (1,-.25);
\draw [fill=white, white] (1.57,.35) circle [radius=0.1];

\draw[thick] (1,.8333) to [out=270, in=90] (2,0);

%%%third twist region

%\draw [thick] (1,3.25-6) to  [out=260,in=100] (1,.25-6);
%\draw [fill=white, white] (.98,3-6) circle [radius=0.1];

%\draw [very thick] (1.4,5.45-1.5) to  (1.55, 5.3-1.5); 
\draw[thick] (2,6-6) to [out=270,in=90] (1,4.75-6);
\draw [fill=white, white] (1.5,5.4-6) circle [radius=0.1];

\draw[thick] (1.2,-0.2) to [out=290,in=90] (2,4.5-6);

%\draw[very thick] (1,6.25) to [out=270,in=90] (2,4.5) to [out=270, in=90] (1,3.25);

\draw[thick] (2,4.5-6) to [out=270, in=90] (1,3.25-6);
\draw [fill=white, white] (1.5,3.88-6) circle [radius=0.1];
\draw[thick] (1,4.75-6) to [out=270, in=90] (2,3-6);

%\draw [fill=white, white] (1,.028) circle [radius=0.051];
%\draw [fill=white, white] (1.1,-0.08) circle [radius=0.051];

\end{tikzpicture}

\caption{A 2-bridge knot in square tower form.}\label{fig:2bridgeTower|}
\end{center}

\end{figure}
%%%%%%%%%%%%%

\begin{figure}
\begin{center}
\begin{tikzpicture}[yscale=1]

%\node[inner sep=0pt, above right] (russell) at (0,.2)
   % {\includegraphics[width=3.9in]{PlumbingSquare}};

%\draw[help lines] (0,0) grid (10,6);

%\draw [thick] (0,6) to [out=90, in=90] (1,6.25);

%\draw [fill=white, white] (.98,3) circle [radius=0.1];

%outside plubming square
\draw [thick, rounded corners] (2.5,3.6) to (.4,2) to (7,2) to (9.6,4) to (7.85,4);
\draw [thick, rounded corners] (2.8,3.85) to (3,4) to (3.2,4);
\draw [thick] (3.4,4) to (3.8,4);
\draw [thick] (4.15,4) to (6.3,4);
\draw [thick] (6.5,4) to (6.9,4);
\draw [thick] (7.2,4) to (7.55,4);

%inside plubming sqre
\draw [thick] (6.3,3.45) to (4,3.45) to (2.7, 2.6) to (6.4,2.6) to (7.7,3.45) to (7.1,3.45);
\draw [thick] (6.5, 3.45) to (6.9, 3.45);

%tower 
\draw [thick] (2.7,2.1) to (2.7,4.7) to (4,5.5) to (4,3.1);
\draw [thick] (2.7+3.7,2.1) to (2.7+3.7,4.7) to (4+3.7,5.5) to (4+3.7,3.1);

\draw [thick] (1.7, 3) to (2.6,3);

\draw [thick, rounded corners] (2.8, 3) to (3.3,3) to (3.3,5.07);
\draw [thick, rounded corners] (2.8+3.7, 3) to (3.3+3.7,3) to (3.3+3.7,5.07);

\draw [thick, rounded corners] (6.3, 3) to (3.3,3) to (3.3,5.07);
\draw [thick, rounded corners] (8.3, 3) to (3.3+3.7,3) to (3.3+3.7,5.07);

\draw [thick] (2.7,1.9) to (2.7,.6) to (6.4,.6) to (6.4,1.9);
\draw [thick] (4.7, .6) to (4.7,1.9);

\draw [thick, rounded corners] (4.7, 2.1) to (4.7,2.65) to (3.7,2);
\draw [thick, rounded corners] (4.7, 2.1) to (4.7,2.65) to (5.9,3.45) to (5.9,3.1);
\draw [thick, rounded corners] (5.9,3.1) to (5.9,3.45) to (6.3,3.75);
\draw [thick] (6.5, 3.88) to (6.67,4);

\draw [thick] (4, 2.9) to (4, 2.7);
\draw [thick] (4, 2.5) to (4, 2.3);
\draw [thick] (4, 1.9) to (4, 1.4) to (4.6, 1.4);
\draw [thick] (5.9, 1.9) to (5.9, 1.4) to (6.3, 1.4) to (4.8, 1.4);
\draw [thick] (5.9, 2.1) to (5.9, 2.5);
\draw [thick] (5.9, 2.7) to (5.9, 2.9);

\draw [thick] (6.5, 1.4) to (7.7, 1.4) to (7.7,2.4);
\draw [thick] (7.7, 2.6) to (7.7, 2.9);

\end{tikzpicture}

\caption{A plumbing square, with $n$ vertical sheets, $n_i$ sheets in the inner plumbing square, and $n-n_i$ sheets in the outer square.}\label{fig:plumbSquare}
\end{center}
\end{figure}
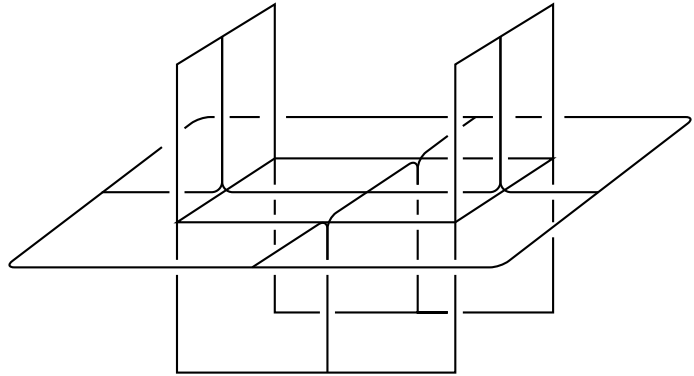
%%%%%%%%%%%%%

\begin{figure}[htbp]
\begin{center}
\includegraphics[height=2.2in]{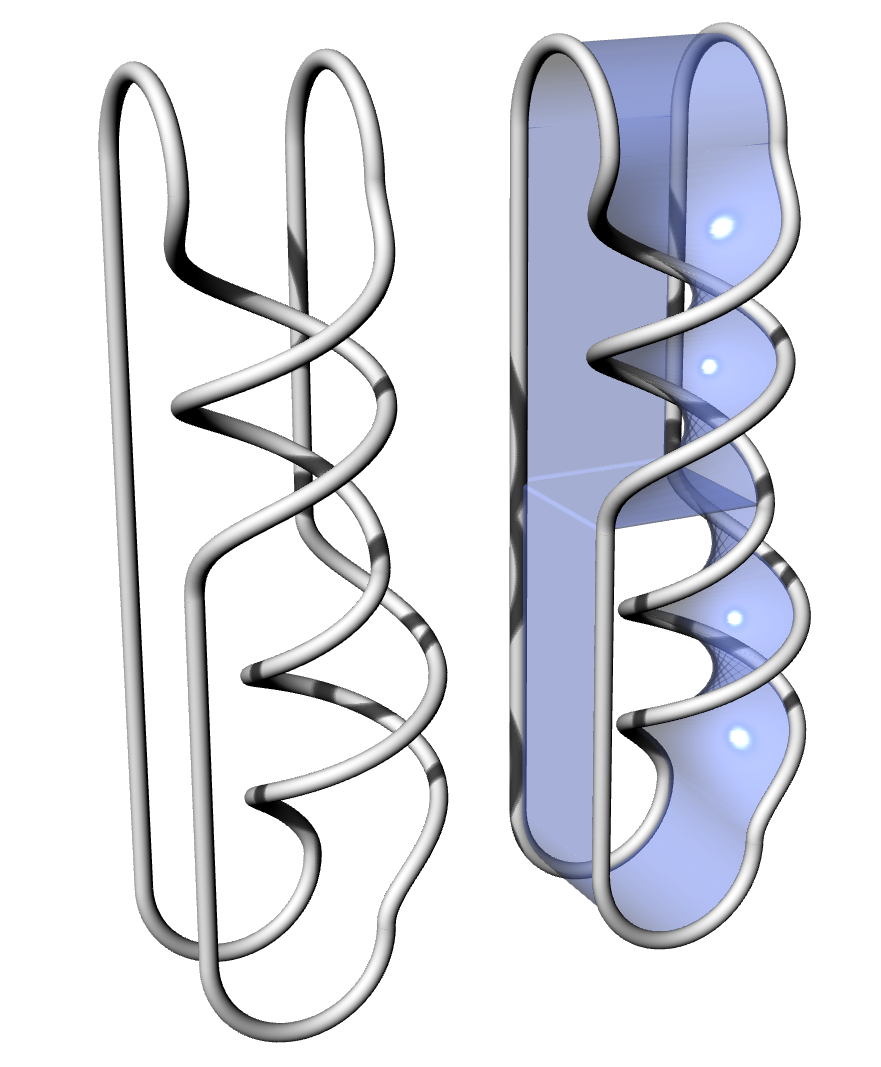}

\end{center}\caption{On the left, the figure 8 knot, and on the right the same knot with the surface $S_1(1)$.}\label{fig:figure8Asurface}
\end{figure}

\begin{theorem}\label{thm:HTthm}{{\cite[Theorem 1]{HT}}}
Let $\frac{p}{q}$ be a rational number with continued fraction decomposition $r+[b_1,b_2, \ldots, b_k]$. 
\begin{enumerate}
\item A closed incompressible surface in $S^3-K_{p/q}$ is a torus isotopic to the boundary of a tubular neighborhood of $K_{p/q}$.

\item A non-closed incompressible, $\partial$-incompressible surface in $S^3-K_{p/q}$ is isotopic to one of the surfaces $S_n(n_1,\ldots,n_{k-1})$ carried by $\Sigma [b_1,\ldots,b_k]$, for some continued fraction expansion $p/q=r+[b_1,\ldots,b_k]$ with $\vert b_i \vert \geq 2$ for each $i$.

\item The surface $S_n(n_1,\ldots,n_{k-1})$ carried by $\Sigma [b_1,\ldots,b_k]$ is incompressible and $\partial$-incompressible if and only if $\vert b_i \vert \geq 2$ for each $i$.

\item Surfaces $S_n(n_1,\ldots,n_{k-1})$ carried by distinct $\Sigma [b_1,\ldots,b_k]$'s with $\vert b_i \vert \geq 2$ for each $i$ are not isotopic.

\item The relation of isotopy among the surfaces $S_n(n_1,\ldots,n_{k-1})$ carried by a given $\Sigma [b_1,\ldots,b_k]$ with $\vert b_i \vert \geq 2$ for each $i$ is generated by: 

$S_n(n_1,\ldots,n_{i-1}, n_i ,\ldots, n_{k-1})$  is isotopic to \\
$S_n(n_1,\ldots,n_{i-1}+1, n_i+1 ,\ldots, n_{k-1})$ if $b_i = \pm 2$. (When $i=1$ this means $S_n(n_1,n_2,\ldots,n_{k-1})$ is isotopic to 
$S_n(n_1+1,n_2\ldots,n_{k-1})$, and similarly when $i=k$.)

\end{enumerate}
\end{theorem}

\begin{remark}\label{rmk:HT}
The main points of this theorem that we will use are $(1)$ and $(2)$, which gives us that the only closed incompressible surfaces in the complement of a 2-bridge knot $K_{p/q}$ are isotopic to the boundary of $E(K_{p/q})$ and that every incompressible surface with boundary is isotopic to some $S_n(n_1,\ldots,n_k)$.
\end{remark}
%%%%%%%%%%%
%Stairstep begin

\subsection{Tunnel Number}\label{sec:tunnel}

Recall $E(K)$ is the exterior of a knot. A family of mutually disjoint properly embedded arcs $\Gamma$ in $E(K)$, the exterior of the knot $K$, is said to be an {\em unknotting tunnel system} if $E(K)-\eta(\Gamma)$ is homeomorphic to a handlebody. 

\begin{definition}\label{def:tunnel}
The {\em tunnel number} of a knot $K$, $t(K)$, is the minimum number of arcs in an unknotting tunnel system, over all unknotting tunnel systems for $K$. 

\end{definition}
 
Morimoto gives an equivalent definition in \cite{Morimoto}, which we will use here. For a knot $K$ in $S^3$, there is a Heegaard splitting $(V_1, V_2)$ of $S^3$ such that a handle of $V_1$ contains $K$ as a core of $V_1$.

\begin{definition}\label{def:Mtunnel}

The minimum genus of $V_1$ minus one, over all Heegaard splittings satisfying the above fact, is the tunnel number, $t(K)$. 

\end{definition}

%%%%%%%%%%%%%%%%%%%%%%%%%%%%%%%%%%%%%%

\subsection{Distance}\label{sec:dist}

For a more thorough discussion on distance, see \cite{Tomova}. Given a compact, orientable, properly embedded surface $S$ in a 3-manifold $M$, the 1-skeleton of the {\em curve complex}, $\mathcal{C}(S)$, is the graph whose vertices correspond to isotopy classes of essential simple closed curves in $S$ such that two vertices are connected if the corresponding isotopy classes have disjoint representatives. For two subsets $A$ and $B$ of $\mathcal{C}(S)$, the {\em distance} between them, $d(A,B)$ is defined to be the length of the  shortest path from an element of $A$ to an element of  $B$.  

For any subset $X\subset S^3$, let $X_K$ be $E(K)\cap X$ .

\begin{definition}\label{def:dist}{{\cite{Tomova}}} Suppose $M$ is a closed, orientable irreducible 3-manifold containing a knot $K$ and suppose $P_K$ is a bridge surface for $K$ splitting $M$ into handlebodies $V$ and $W$. Let $\mathcal{V}$ (resp $\mathcal{W}$) be the set of all essential simple closed curves on $P_K$ that bound disks in $V_K$ (resp. $W_K$). Then the {\em distance} $d(P, K): = d(\mathcal{V}, \mathcal{W})$ measured in $C(P_K )$.

\end{definition}

We use this definition in Section \ref{sec:BSMontandPret}  as noted in the introduction to show that the converse of Theorem \ref{thm:dist} does not hold.  A note on intuition, one usually thinks of distance as an invariant which describes how interconnected the bridge disks are to each other.  A knot that has a high distance genus zero bridge surface would imply that the bridge disks are all ``overlapping" and thus meridional stabilization is the best we can do, and thus, the bridge spectrum is stair-step.

%%%%%%%%%%%%%%%%%%%%%%%%%%%%%%%%%%%%%%

\subsection{Generalized Montesinos Knots}\label{subsec:GMK}

J. Montesinos defined the class of knots and links that now bear his name in 1973 in \cite{Montesinos}.  As stated earlier on page \pageref{2-bridge knot}, given a rational number $\frac{\beta}{\alpha} \in \Q$, there is a unique continued fraction decomposition $\frac{\beta}{\alpha}=[a_1,a_2, \ldots, a_m] $ where $a_i \not = 0$ for all $i=1,\ldots,m$ and $m$ is odd. We also recall that for each rational number, there is an associated rational tangle, see page \pageref{2bTangle} and Figure \ref{fig:rationaltangle}.

\begin{definition} 
A {\em Montesinos knot or link} $M(\frac{\beta_1}{\alpha_1},\frac{\beta_2}{\alpha_2},\ldots,\frac{\beta_n}{\alpha_n}\vert e)$ is the knot in Figure \ref{fig:mont} where each $\beta_i,\alpha_i$ for $i=1,\ldots, n$ represents a rational tangle given by $\frac{\beta_i}{\alpha_i}$, and $e$ represents the number of positive half-twists. If $e$ is negative, we have negative half-twists instead. 

\end{definition}

\begin{figure}
\begin{center}

\begin{tikzpicture}[xscale = .7, yscale = .7]

%Montesinos link

%\draw[help lines] (0,0) grid (15,20);

%\draw (1,1) rectangle +(10,15);
%\draw (2,2) rectangle +(8,13);

\draw (9.5,3) rectangle +(2,2);

\draw (9.5,9) rectangle +(2,2);
\draw (9.5,12) rectangle +(2,2);

\draw [domain=0:.4] plot ({.5*cos(pi* \x r )+1.5}, \x+6);

\draw [domain=0:1.4] plot ({- .5*cos(pi* \x r )+1.5}, \x+6);

\draw [domain=.6:2.4] plot ({.5*cos(pi* \x r )+1.5}, \x+6);

\draw [domain=.6:2.4] plot ({.5*cos(pi* \x r )+1.5}, \x+7);

\draw [domain=.6:2] plot ({.5*cos(pi* \x r )+1.5}, \x+8);
\draw [domain=.6:1] plot ({.5*cos(pi* \x r )+1.5}, \x+9);

\draw [rounded corners] (1,6) -- (1,1) -- (11,1)--(11,3);
\draw [rounded corners] (2,6) -- (2,2) -- (10,2)--(10,3);

\draw [rounded corners] (1,10) -- (1,16) -- (11,16)--(11,14);
\draw [rounded corners] (2,10) -- (2,15) -- (10,15)--(10,14);

\foreach \a in {5,8,11}{
	\draw (10,\a) -- (10,\a+1);
	\draw (11,\a) -- (11,\a+1);
	}

\draw [fill] (10.5,6.7) circle [radius=0.05];	
\draw [fill] (10.5,7) circle [radius=0.05];	
\draw [fill] (10.5,7.3) circle [radius=0.05];	

\node at (0.5,8) {$e$};

\node at (10.5, 13) {$\beta_1,\alpha_1$};
\node at (10.5, 10) {$\beta_2,\alpha_2$};
\node at (10.5, 4) {$\beta_n,\alpha_n$};

\end{tikzpicture}
\caption{The Montesinos link $M(\frac{\beta_1}{\alpha_1},\frac{\beta_2}{\alpha_2},\ldots,\frac{\beta_n}{\alpha_n}\vert e)$}\label{fig:mont}
\end{center}

\end{figure}
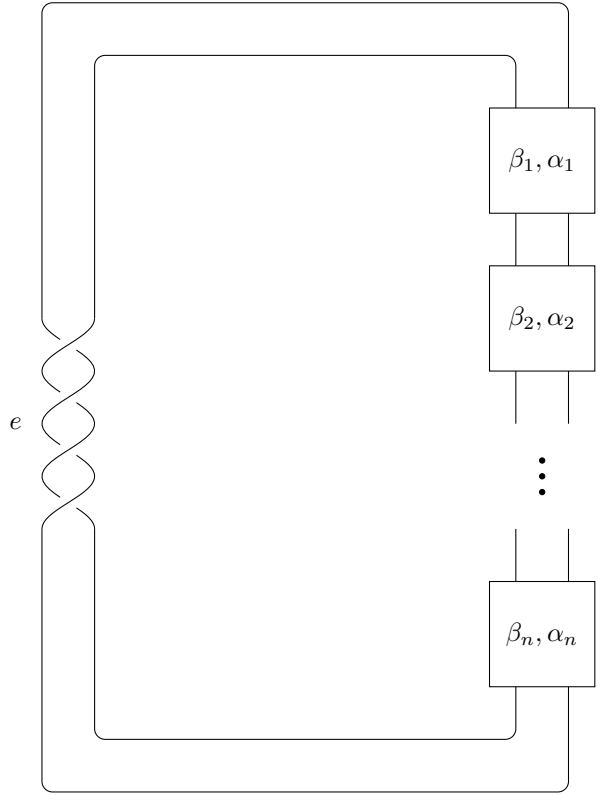

Lustig and Moriah in \cite{LM} defined the class of knots which  we describe in the rest of this subsection.  They based their definition off of Boileau and Zieschang \cite{BZ}, who prove that any Montesinos knot $M(\frac{\beta_1}{\alpha_1},\frac{\beta_2}{\alpha_2},\ldots,\frac{\beta_n}{\alpha_n}\vert e)$, which does not have integer tangles, i.e., $\alpha_i \not = 1$ for all $i$, has bridge number $n$. Consider Figure \ref{fig:gmk}. Each $\alpha_{i,j},\beta_{i,j}$ is the 4-plat diagram from the rational tangle defined by $\alpha_{i,j}/\beta_{i,j}$. An $n$-braid is $n$ disjoint arcs in $I^3$ with initial points in $I\times\{\frac{1}{2}\} \times \{1\}$, end points on $I\times\{\frac{1}{2}\} \times \{0\}$, and the arcs strictly decreasing in the third component of $I^3$. A double of a braid is obtained by duplicating each arc in an $\epsilon$-neighborhood of the original, possibly with twisting of an arc and its duplicate, so it becomes a $2n$-braid.  In a generalized Montesinos knot, each $B_j$ is a double of an $n$-braid. For a more in depth description, see \cite{BZ} and \cite{LM}. They exhibit a number diagrams which show that every Montesinos Knot and in particular, every pretzel knot, is a Generalized Montesinos Knot. One should note that for a pretzel knot $K_n=K_n(p_1,\ldots, p_n)$, the corresponding rational tangles are $p_k=\alpha_{i,j}/\beta_{i,j}$, which are integers. 

\begin{definition}\label{def:genMont}
A {\em generalized Montesinos knot or link} \newline $$K=M\left( \{(\beta_{i,j},\alpha_{i,j})\}_{i=1,j=1}^{i=\ell,j=m}, \{B_i\}_{j=1}^{j=\ell -1} \right)$$ is the knot in Figure \ref{fig:gmk}, where each $\beta_{i,j},\alpha_{i,j}$ for $i=1,\ldots, m$ and $j=1,\ldots, \ell$ represents a rational tangle given by $\frac{\beta_{i,j}}{\alpha_{i,j}}$ and $B_i$ represents a $2n$-braid, which is obtained by doubling an $n$-braid.

\end{definition}

The main result from Lustig and Moriah's paper that we use below is the following.  Recall that $rk(G)$, the rank of the group $G$, is the minimum number of generators; the notation $t(K)$ is the tunnel number, see Definition \ref{def:tunnel}; and $b(K)$ is the bridge number, which we are usually denoted as $b_0(K)$, the genus zero bridge number, see Definition \ref{def:genus $g$ bridge number}.

\begin{theorem} \label{LMt}{{\cite[Theorem 0.1]{LM} }}
Let $K$ be a generalized Montesinos knot/link as in Figure \ref{fig:gmk} below, with $2n$-plats. Let  $\alpha=\text{gcd}(\alpha_{i,j}:i=1,\ldots,\ell; j=1,\ldots,m)$.  If $\alpha \not = 1$ then rk$(\pi_1(S^3-K))=t(K)+1=b(K)=n$.
\end{theorem}

\newcount \mycount

\begin{figure}
 \begin{center}

 \begin{tikzpicture}[transform shape]
 
 %Generalized Montesinos link
 
%\draw[help lines] (0,0) grid (15,20);

\draw (0,3.5) rectangle +(13,2);
\draw (0,7.5) rectangle +(13,2);
\draw (0,12.5) rectangle +(13,2);

%\draw (0,1) rectangle +(2,2) ;
%\draw (3,1) rectangle +(2,2) ;
%\draw (6,1) rectangle +(2,2) ;
%\draw (11,1) rectangle +(2,2) ;

\node at (1,2) {$\alpha_{\ell,1},\beta_{\ell,1}$};
\node at (4,2) {$\alpha_{\ell,2},\beta_{\ell,2}$};
\node at (7,2) {$\alpha_{\ell,3},\beta_{\ell,3}$};
\node at (12,2) {$\alpha_{\ell,m},\beta_{\ell,m}$};

\node at (7,4.5) {$B_{\ell-1}$};
\node at (7,8.5) {$B_{2}$};
\node at (7,13.5) {$B_{1}$};

\foreach \d [count = \di] in {16,11}{
	\node at (12,\d) {$\alpha_{\di,m},\beta_{\di,m}$};
	
	\foreach \e [count = \ei] in {0,3,6}{
		\node at (\e+1,\d) {$\alpha_{\di,\ei},\beta_{\di,\ei}$};

		}
	}

\foreach \a [count = \ai] in {0,3,6,11}{
	\draw (\a,1) rectangle +(2,2);
	\draw (\a,10) rectangle +(2,2);
	\draw (\a,15) rectangle +(2,2);

	\draw [domain=180:360] plot ({\a+.35+.15*cos(\x)}, {.5+.15*sin(\x)});	
	\draw [domain=180:360] plot ({\a+1.65+.15*cos(\x)}, {.5+.15*sin(\x)});

	\draw [domain=0:180] plot ({\a+.35+.15*cos(\x)}, {17.5+.15*sin(\x)});	
	\draw [domain=0:180] plot ({\a+1.65+.15*cos(\x)}, {17.5+.15*sin(\x)});

	\foreach \b in {.5,3,5.5,7,9.5,12,14.5,17}{
		\draw (\a+.2,\b) -- (\a+.2,\b+.5);
		\draw (\a+.5,\b) -- (\a+.5,\b+.5);

		\draw (\a+1.8,\b) -- (\a+1.8,\b+.5);
		\draw (\a+1.5,\b) -- (\a+1.5,\b+.5);
	}
	
	}
\foreach \c in {2,11,16}{
	
	\draw [fill] (9,\c) circle [radius=0.05];	
	\draw [fill] (9.5,\c) circle [radius=0.05];	
	\draw [fill] (10,\c) circle [radius=0.05];	
	}
	
\draw [fill] (7,6.2) circle [radius=0.05];	
\draw [fill] (7,6.5) circle [radius=0.05];	
\draw [fill] (7,6.8) circle [radius=0.05];

   %\draw [domain=180:360] plot ({cos(\x)}, {sin(\x)});	

\end{tikzpicture}

\caption{A generalized Monetesinos knot}\label{fig:gmk}
\end{center}
\end{figure}
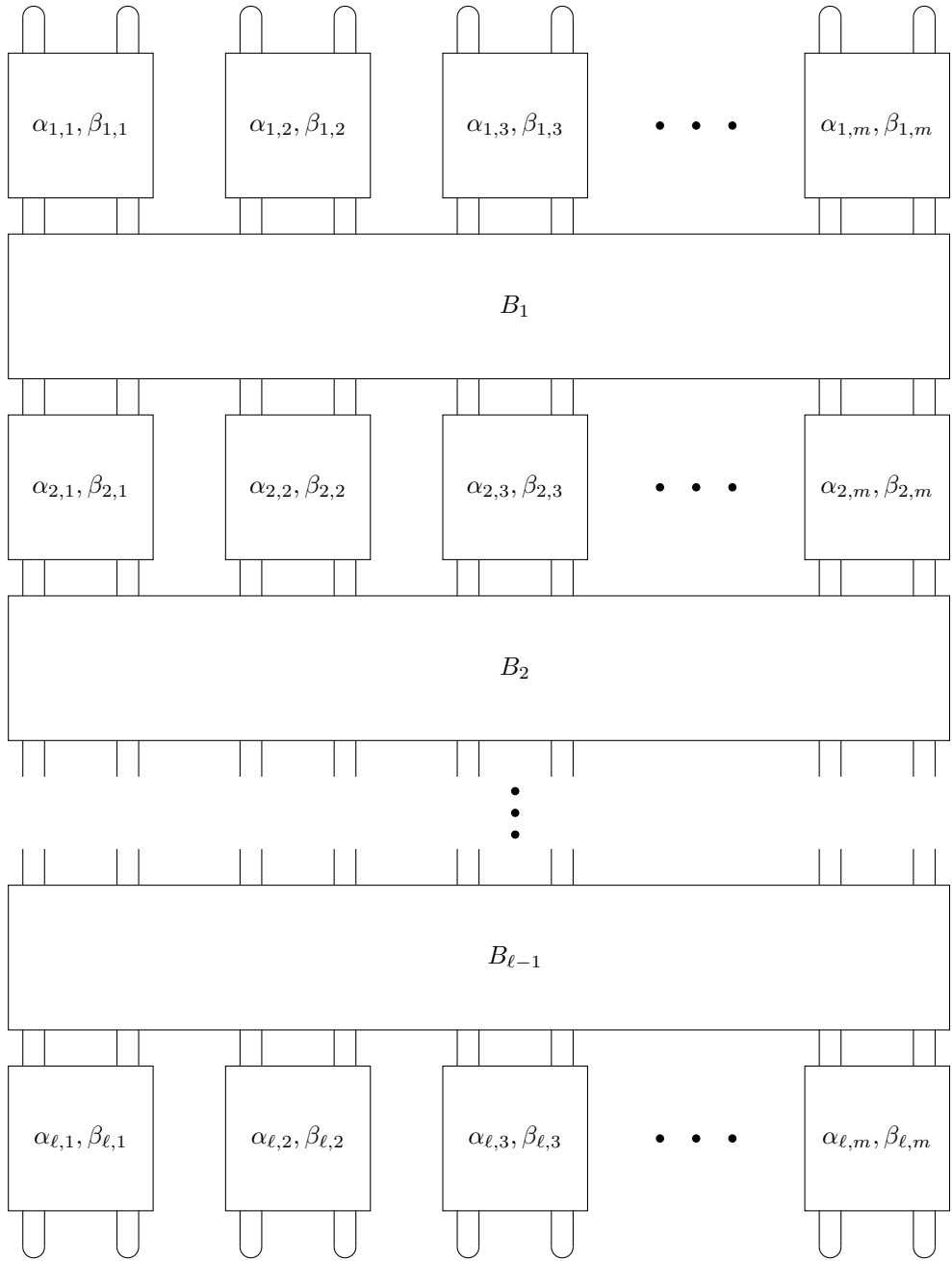

%stairstep end
%%%%%%%%%%%%%

\subsection{Results on bridge surfaces} 
In \cite{Z1}, Zupan introduced bridge spectrum and in the same paper produced the results which appear in this subsection. 

Let $M$ be a 3-manifold with boundary, let $P$ be a subsurface of $\p M$, and let $A$ be a properly embedded surface in $M$. Then a {\em $P$-$\p$-compressing disk} for $A$ is a $\p$-compressing disk $\Delta$ for $A$ with the added condition that $\Delta \cap \p M \subset P$.  Also, $A$ is {\em $P$-essential} if $A$ is incompressible and there does not exist a $P$-$\p$-compressing disk for $A$ in $M$. Similarly, $A$ is {\em $P$-strongly irreducible} if $A$ is separating and admits either compressing or $P$-$\p$-compressing disks on either side but admits no pair of disjoint disks on opposite sides.  Two surfaces $A$ and $B$ are {\em almost transverse} if $A$ is transverse to $B$ except for a single saddle tangency. The following lemmas of Zupan will be used in sections 3 to 5 below. For the following lemma, recall for a 3-manifold $M$ and an embedded 1-manifold $J\subset M$, $M(J)$ denotes $M-\eta(J)$ and for an embedded surface $\Sigma\subset M$, $\Sigma_J$ denotes $\Sigma - \eta(J)$.

\begin{lemma} \label{Z5.2}{{\cite[Lemma 5.2]{Z1}}} Let $M$ be a compact 3-manifold and $J$ a properly embedded 1-manifold, with $Q:=\partial N(J)$ in $M(J)$.  Suppose $\Sigma$ is a strongly irreducible bridge splitting surface for $(M,J)$, and let $S\subset M(J)$ be a collection of properly embedded essential surfaces such that for each component $c$ of the boundary of each element of $S$, either $c\subset Q$ or $c\subset \partial M$.  Then one of the following must hold:

\begin{enumerate}
\item After isotopy, $\Sigma_J$ is transverse to each element of $S$ and each component of $\Sigma_J \setminus \eta(S)$ is $Q$-essential in $M(J) \setminus \eta(S)$.
\item After isotopy, $\Sigma_J$ is transverse to $S$, one component of $\Sigma_J \setminus \eta(S)$ is $Q$-strongly irreducible and all other components are $Q$-essential in $M(J) \setminus \eta(S)$.
\item After isotopy, $\Sigma_J$ is almost transverse to $S$ and each component of $\Sigma_J \setminus \eta(S)$ is $Q$-essential in $M(J) \setminus \eta(S)$.
\end{enumerate}

\end{lemma}

\begin{lemma} \label{lem:Zlem6.1}{{\cite[Lemma 6.1]{Z1}}} Let $J$ be a knot in a 3-manifold $M$ and let $K={\mathsf {cable}}(T_{m,n}, J)$ be a $(m,n)$-cable of $J$.  If $\Sigma \subset M$ is a Heegaard surface such that $J\subset \Sigma$ and if there is a compressing disk $D$ for $\Sigma$ such that $\vert D\cap J \vert =1$, then there exists an embedding of $K$ in $M$ such that $K\subset \Sigma$. 

\end{lemma}

This lemma tells us that for $K_{p/q}$ a 2-bridge knot, $K={\mathsf {cable}}(T_{m,n}, K_{p/q})$ satisfies $b_2(K)=0$, since every $K_{p/q}$ can be isotoped into a genus 2 Heegaard surface with each handle having a single arc of the knot traversing it.

%%%%%%%%%%%%%%%%%%%%%%%%%%%%%%%%%%%%%%%%%%%%%%%%%%

%%%%%%%%%%%%%%%%%%%%%%%%%%%%%%%%%%%%%%%%%%%%%%%%%%%%%%%%%%%%%%%%%%%%%%%%%%%%
%%%%%%%%%%%%%%%%%%%%%%%%%%%%%%%%%%%%%%%%%%%%%%%%%%%%%%%%%%%%%%%%%%%%%%%%%%%%

%%%%%%%%%%%%%%%%%%%%%%%%%%%%%%%%%%%%%%%%%%%%%%%%%%%%%%%%%%%%%%%%%%%%%%%%%%%%
 
\section{Results about surfaces}\label{sec:Alex}

In this section we introduce and prove some technical results needed for the proof of Theorem \ref{thm:BSofC2b}.

\begin{lemmass}\label{lem:eulerofessent}
The surface $S_n(n_1,\ldots,n_{k-1})$ has Euler characteristic $-n(k-1)$.

\end{lemmass}

\begin{proof}

For a 2-bridge knot $K_{p/q}$, with $p/q=r+[b_1,\ldots,b_k]$, consider the single sheeted surface $S_1(n_1,\ldots,n_{k-1})$, see Section \ref{sec:HT}.  The vertical bands deformation retract to an edge and the plumbing squares deform to a vertex. As our knot lives in $S^3$, the outer squares can be deformed to a point at infinity.  See Figure \ref{fig:deformgraph}. Thus, we get a graph with $k-1$ vertices and $2(k-1)$ edges. The Euler characteristic of this graph and surface is $-(k-1)$. Given an $n$-sheeted surface,  $S_n(n_1,\ldots,n_{k-1})$, there are $n$ copies of this graph, not taking into account how these graphs are connected.  But, this is only a change in adjacency, not in the number of vertices and edges, so it has no effect on the Euler characteristic.  Thus, $\chi \left( S_n(n_1,\ldots,n_{k-1}) \right)=-n(k-1)$.\end{proof}

%%%%%%%%%%%%%%%%%%%%%%%%%%%%%%%%%%%%%%%%%%%%%%%%%%

% A 2-bridge knot vertical 

\begin{figure}[h]
\begin{center}
\begin{tikzpicture}[yscale=1]

%\draw[help lines] (0,0) grid (4,7);

\draw (1,5) to  [out=120,in=180] (1,6) to  [out=0,in=60] (1,5) ;
\draw [fill] (1,5) circle [radius=0.08];
\node [left] at (1,5) {$v_1$};

\draw (1,5) to  [out=240,in=120] (1,4);
\draw (1,5) to  [out=300,in=60] (1,4);
\draw [fill] (1,4) circle [radius=0.08];
\node [left] at (1,4) {$v_2$};

\draw (1,4) to  [out=240,in=90] (.85,3.5);
\draw (1,4) to  [out=300,in=90] (1.15,3.5);

\draw [fill] (1,3.2) circle [radius=0.04];
\draw [fill] (1,3) circle [radius=0.04];
\draw [fill] (1,2.8) circle [radius=0.04];

\draw (1,2) to  [out=120,in=270] (.85,2.5);
\draw (1,2) to  [out=60,in=270] (1.15,2.5);
\draw [fill] (1,2) circle [radius=0.08];
\node [left] at (1,2) {$v_{k-1}$};

\draw (1,2) to  [out=240,in=180] (1,1) to  [out=0,in=300] (1,2) ;

\end{tikzpicture}

\caption{The deformation retract of  $S_1(n_1,\ldots,n_{k-1})$.}\label{fig:deformgraph}
\end{center}
\end{figure}
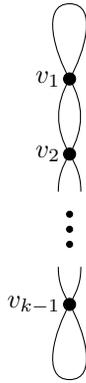

From this lemma, together with Theorem \ref{thm:HTthm} part (2), we see that all non-closed incompressible surfaces in the exterior of any non-torus 2-bridge knot have negative Euler characteristic.  This is because $k$ represents the number of steps in the partial fraction decomposition, which in turn corresponds to the number of twist regions in the vertical diagram of the 2-bridge knot, see Figure \ref{fig:figure8Asurface}.  So, for a 2-bridge knot $K_{p/q}$, $k$ is zero if and only if $K_{p/q}$ is the unknot and $k$ is one if and only if $K_{p/q}$ is a torus knot. This proves the follow lemma.

\begin{lemmass}\label{lem:eulerneg}

The surfaces $S_n(n_1,\ldots,n_k)$ of a non-torus 2-bridge knot have negative Euler characteristic.

\end{lemmass}

Another important property of 2-bridge knots is that they are meridionally small.  

\begin{definitionss}\label{def:meridianallySmall}

A knot $K\subset S^3$ is {\em meridionally small} if $E(K)$ contains no essential surface $S$ with $\partial S$ consisting of meridian curves of $N(K)$.

\end{definitionss}

\begin{lemmass}\label{lem:2bridge are msmall}
Every 2 bridge knot is meridionally small.
\end{lemmass}

\begin{proof}
We notice from Remark \ref{rmk:HT} that all essential surfaces in $E(K)$ either are disjoint from $\partial E(K)$ or intersect $\partial E(K)$ with integer slope. In either case, they do not intersect the boundary in meridian curves. 
\end{proof}

Zupan, in \cite{Z2}, proves the following theorem.

\begin{theoremss}\label{thm:cableMeridianSmall}{{[Theorem 6.6], \cite{Z2}}}
For a knot $J$ in $S^3$, the following are equivalent:
\begin{enumerate}
\item J is meridionally small.
\item Every cable of J is meridionally small.
\item There exists a cable of J that is meridionally small.
\end{enumerate}\end{theoremss}

The following is a modification of Lemma 3.5 from \cite{Z1}.

\begin{lemmass}\label{lem:Z3.5}
Let $K={\mathsf {cable}}(T_{m,n}, K_{p/q})$ be a cable of a 2-bridge knot.  Suppose $S\subset E(K)$ is an essential surface.  If $S$ is not isotopic to the boundary torus of $K_{p/q}$, then $S\cap \partial E(K)$ is nonempty and has integral slope.
\end{lemmass}

\begin{proof}
Let $T$ denote the boundary torus of $E(K_{p/q})$.  Recall that $C_{m,n}$ is a cable space, see page \pageref{def:cablespace}.  After isotopy, we may assume $\vert S\cap T\vert$ is minimal.  

\noindent {\em Claim}: Each component of $S\cap C_{m,n}$ is incompressible in $C_{m,n}$ and each component of $S\cap E(K_{p,q})$ is incompressible in $E(K_{p,q})$.  

To prove this claim, let $D$ be a compressing disk for $S\cap C_{m,n}$ in $C_{m,n}$.  Then $\partial D$ bounds a disk $D'\subset S$ by the incompressibility of $S$, where $D'\cap T\not = \emptyset$. By the irreducibility of $E(K)$, there is an isotopy of $S$ pushing $D'$ onto $D$ which reduces $\vert S\cap T\vert$, yielding a contradiction.  The argument is similar for $S\cap E(K_{p,q})$ being incompressible in $E(K_{p,q})$, proving the claim.

If $S\cap T = \emptyset$, then $S\subset C_{m,n}$ or $S\subset E(K_{p,q})$. If $S\subset C_{m,n}$, then by our claim, each component of $S\cap C_{m,n}$ is incompressible in $C_{m,n}$, and thus $S$ is closed and vertical and the only such surfaces are boundary parallel by \cite{Wald}.  Similarly, if $S\subset E(K_{p,q})$, then by our claim, each component of $S\cap E(K_{p,q})$ is incompressible in $E(K_{p,q})$ and then by Remark \ref{rmk:HT}, $S$ is closed and thus boundary parallel.  Thus, $S$ is isotopic to $T$. 

If $S\cap T \not= \emptyset$, then $S\cap E(K_{p,q})$ is one of the surfaces classified by Theorem \ref{thm:HTthm} and so $S\cap E(K_{p,q})$ has integer slope.  The surface $T$ is also $\p_+C_{m,n}$.   Then we apply Lemma \ref{lem:Z3.2} which tells us that $S$ has integer slope in $\p_-C_{m,n}$, thus $S\cap \p E(K)$ has integer slope. \end{proof}

%%%%%%%%%%%%
%Stairstep begin

\section{Bridge spectra of Montesinos and pretzel knots}\label{ch:MontAndPret}

%%%%%%%%%%%%%%%%%%%%%%%%%%%%%%%%%%%%%%%%%%%%%%%%%%%%%%%%%%%%%%%%%%%%%%%%%%%%
%%%%%%%%%%%%%%%%%%%%%%%%%%%%%%%%%%%%%%%%%%%%%%%%%%%%%%%%%%%%%%%%%%%%%%%%%%%%

Before determining the bridge spectra in subsection \ref{sec:BSMontandPret}, we begin subsection \ref{sec:TunnelBridge} with further results on the relationship between tunnel number and bridge spectrum.

\subsection{Tunnel number and bridge spectrum}\label{sec:TunnelBridge}
The next proposition appears in \cite{MSY} without proof; we provide one here for completeness. Also, recall a knot is primitive in a $(g,b)$-splitting if it transversely intersects the boundary of a properly embedded essential disk of the handlebody in a single point. See subsection \ref{sec:bridgespectrum}.

\begin{proposition} \label{prop:MSYsprop}
If $K$ is a knot with a $(g,b)$-splitting satisfying either $b>0$, or $b=0$ and $K$ is primitive, then $t(K)\leq g+b-1$.

\end{proposition}

\begin{proof}

Recall Definition \ref{def:Mtunnel}, the tunnel number of a knot $K$ is the minimum genus minus one over all Heegaard surfaces with that contain $K$ as its core.  Suppose first $b>0$; we can meridionally stabilize each of the $b$ arcs, see page \pageref{def:MeridionalStabilize}.  This takes the genus $g$ surface and adds in $b$ more handles, which gives us a $g+b$ genus surface. This is precisely the situation in Definition \ref{def:Mtunnel}. Hence, the new Heegaard surface is of genus $g+b$ with $K$ as the core. Thus the tunnel number is at most $g+b-1$. 

If $b=0$ and the knot is primitive, then by definition, there is an essential disk for the Heegaard surface which intersects the knot exactly once. This essential disk  must be a meridian disk and thus, the knot must have a tunnel number that is at most $g-1$. 
\end{proof}

This is the driving force behind the following proposition:

\begin{proposition} \label{prop:stairstep}
If $K$ is a knot $K$ with $t(K)+1=b(K)=b_0(K)$, then $K$ has the stair-step primitive bridge spectrum, ${\mathbf {\hat b}}(K)=(\hat{b}_0(K),\hat{b}_0(K)-1,\ldots,2,1,0)$.
\end{proposition}

\begin{proof}

By Proposition \ref{prop:MSYsprop}, for any $(g,b)$-splitting,  $t(K)\leq g+b -1$, and so $b_0(K)-1=t(K)\leq g+b_g(K)-1$ for all $g$, and so $b_g(K)\geq b_0(K)-g$. Corollary \ref{cor:boundedabove} states that for every knot $K$, and every $g$, there is a $(g,b_0(K)-g)$-splitting for $K$.  Since $b_g(K)\leq b_0(K)-g$, we get that $b_g(K)=b_0(K)-g$ for all $g$. 
\end{proof}

\subsection{Bridge spectra of Montesinos knots and pretzel knots}\label{sec:BSMontandPret}
\begin{theorem}\label{thm:bsgmk}
A generalized Montesinos knot or link, \newline $K=M\left( \{(\beta_{i,j},\alpha_{i,j})\}_{i=1,j=1}^{i=\ell,j=m}, \{B_i\}_{j=1}^{j=\ell -1} \right)$ with $\alpha =  \gcd \{\alpha_{i,j} \} \not = 1$ has the stair-step primitive bridge spectrum, ${\mathbf {\hat b}}(K)=(\hat{b}_0(K),\hat{b}_0(K)-1,\ldots,2,1,0)$.

\end{theorem}
 
\begin{proof} This is immediate from Proposition \ref{prop:stairstep} and Theorem \ref{LMt}.
\end{proof}

Notice that Theorem \ref{thm:bsgmk} implies that any pretzel knot $K_n=K_n(p_1,\ldots,p_n)$ also has the stair-step primitive bridge spectrum if $\alpha=\text{gcd}(p_1,\ldots,p_n)\not = 1$.  It should be noted that pretzel knots  $K_n$ with $\alpha \not = 1$, via proposition \ref{prop:bsVSpbs}, while having stair-step primitive bridge spectra, have the bridge spectrum ${\mathbf b}(K_n)=(n,n-1,\ldots, 3,2,0)$. This is because any pretzel knot embeds into a genus $n-1$ surface. See, for example, Figure~\ref{fig:pretzelOnSurface}.  These facts prove the following corollary.

\begin{corollary}\label{cor:bsp}
Given a $K_n=K_n(p_1,\ldots,p_n)$ pretzel knot with $\gcd(p_1,\ldots,p_n)\not = 1$, the primitive bridge spectrum is stair-step, i.e. ${\mathbf {\hat b}}(K_n)=(n,n-1,\ldots,2,1,0)$, and ${\mathbf b}(K_n)=(n,n-1,\ldots,3,2,0)$.
\end{corollary}

%%%%%%%%

\begin{figure}\label{fig:pretz_embed}
\begin{tikzpicture}

%pretzel knot embedded on a genus g surface

%\draw[help lines] (0,0) grid (12,8);
\filldraw[fill=lightgray]
 (0,6) to [out=90,in=180] (2,8) -- (12,8) to [out=0, in=90] (14,6)--(14,2) to [out=270, in=0] (12,0) -- (2,0) to [out=180,in=270] (0,2) -- (0,6) --cycle
 
 \foreach \a in {2,6,10}{
(2+\a,5) to [out=90, in=0] (1+\a,6) to [out=180,in=90] (\a,5) -- (\a,3) to [out=270,in=180] (\a+1,2) to [out=0,in=270] (2+\a,3) -- (2+\a,5)--cycle};
% Counter-clockwise rectangle

\draw [thick] (2/3,2.5) to [out=270, in=180] (3,2/3) -- (11,2/3) to [out=0,in=270] (13+1/3,2.5);
\draw [thick] (2/3,5.5) to [out=90, in=180] (3,8-2/3) -- (11,8-2/3) to [out=0,in=90] (13+1/3,5.5);

\foreach \a in {0,4,8}{
\draw [thick] (4/3+\a,2.5) to [out=270, in=180] (3+\a,4/3) -- (3+\a,4/3) to [out=0,in=270] (5-1/3+\a,2.5);
\draw [thick] (4/3+\a,5.5) to [out=90, in=180] (3+\a,8-4/3) -- (3+\a,8-4/3) to [out=0,in=90] (5-1/3+\a,5.5);
}

%First handle
\draw [thick] (4/3,2.5) to [out=90,in=190] (2,3);
\draw [thick, dashed] (2,3) to (0,13/3);
\draw [thick] (0,13/3) to [out=10,in=270] (4/3,5.5);

\draw [thick] (2/3,2.5) to [out=90,in=190] (2,11/3);
\draw [thick, dashed] (2,11/3) to (0,5);
\draw [thick] (0,5) to [out=10,in=270] (2/3,5.5);

%second handle
\draw [thick] (16/3,2.5) to [out=90,in=-10] (4,11/3);
\draw [thick, dashed] (4,11/3) to  (6,13/3);
\draw [thick] (6,13/3) to [out=170,in=270] (14/3,5.5);

\draw [thick] (14/3,2.5) to [out=90,in=-10] (4,3);
\draw [thick, dashed] (4,3) to (6,4-1/3);
\draw [thick] (6,4-1/3) to [out=170,in=-10] (4,4+1/3);
\draw [thick,dashed] (4,4+1/3) to  (6,5);
\draw [thick] (6,5) to [out=170,in=270] (16/3,5.5);

%third handle
\begin{scope}[xshift=4cm]
\draw [thick] (16/3,2.5) to [out=90,in=-10] (4,11/3);
\draw [thick, dashed] (4,11/3) to  (6,13/3);
\draw [thick] (6,13/3) to [out=170,in=270] (14/3,5.5);

\draw [thick] (14/3,2.5) to [out=90,in=-10] (4,3);
\draw [thick, dashed] (4,3) to (6,4-1/3);
\draw [thick] (6,4-1/3) to [out=170,in=-10] (4,4+1/3);
\draw [thick,dashed] (4,4+1/3) to  (6,5);
\draw [thick] (6,5) to [out=170,in=270] (16/3,5.5);
\end{scope}

%fourth handle
\begin{scope}[xscale=-1,xshift=-18cm]
\draw [thick] (16/3,2.5) to [out=90,in=-10] (4,11/3);
\draw [thick, dashed] (4,11/3) to  (6,13/3);
\draw [thick] (6,13/3) to [out=170,in=270] (14/3,5.5);

\draw [thick] (14/3,2.5) to [out=90,in=-10] (4,3);
\draw [thick, dashed] (4,3) to (6,4-1/3);
\draw [thick] (6,4-1/3) to [out=170,in=-10] (4,4+1/3);
\draw [thick,dashed] (4,4+1/3) to  (6,5);
\draw [thick] (6,5) to [out=170,in=270] (16/3,5.5);
\end{scope}

\end{tikzpicture}

\caption{A pretzel knot $K_4(2,-3,-3,3)$ embedding on a genus three surface.}\label{fig:pretzelOnSurface}

\end{figure}
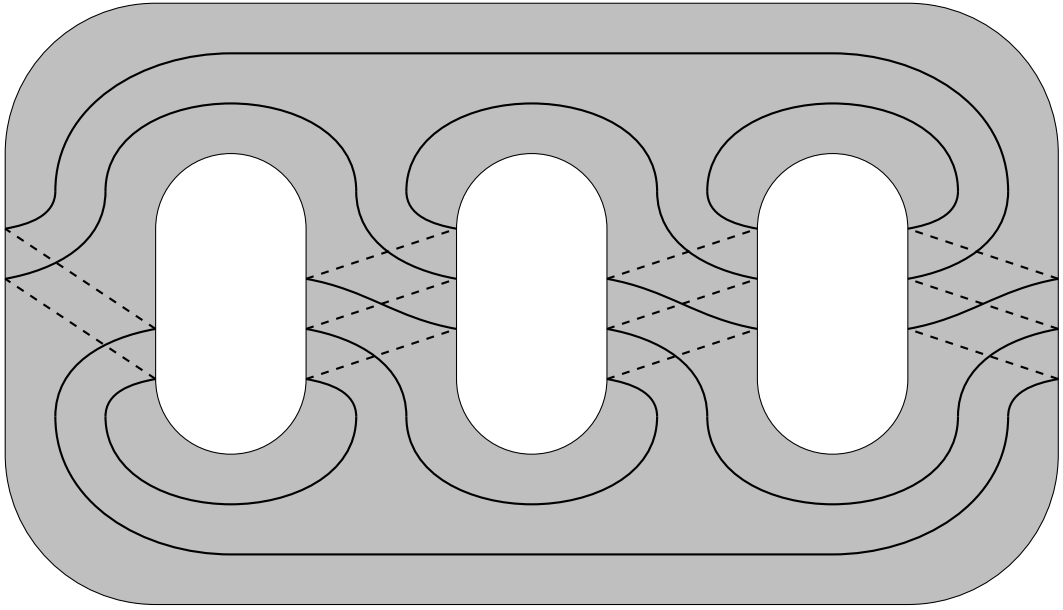

%%%%%%%%

One should ask what happens when $\alpha=1$, to which we do not have a complete answer here. Morimoto, Sakuma, and Yokota in \cite{MSY} show that there are Montesinos knots $K$ which have $\alpha =1$ and $t(K)+2=b(K)$.  Here is their theorem as stated by Hirasawa and Murasugi in \cite{HM}:

\begin{theorem} \label{MSYtunnel}{{\cite{HM}}} The Montesinos knot $K=M(\frac{\beta_1}{\alpha_1},\frac{\beta_2}{\alpha_2},\ldots,\frac{\beta_r}{\alpha_r}\vert e)$ has tunnel number one if and only if one of the following conditions is satisfied.

\begin{enumerate}
\item $r=2.$
\item $r=3, \beta_2/\alpha_2 \equiv \beta_3/\alpha_3 \equiv \pm \frac{1}{3}$ in $\Q/ \Z$, and $e+\sum^3_{i=1}  \beta_i/\alpha_i = \pm 1/(3\alpha_1)$.
\item $r=3, \alpha_2$ and $\alpha_3$ are odd, and $\alpha_1=2$.

\end{enumerate}
\end{theorem}

Note that any pretzel link $K_3(p_1,p_2,2)$ is a knot if and only if $p_1$ and $p_2$ are odd.   Every pretzel knot described in Theorem \ref{MSYtunnel} has tunnel number one but bridge number three.  Thus, we cannot use Proposition \ref{prop:stairstep} to compute the bridge spectrum. It is unknown if there is a pretzel knot or Montesinos knot which has bridge spectrum $(3,1,0)$.

Another point of interest is that any $n$-pretzel knot $K_n$ with $n\geq4$ is a distance one knot.

\begin{proposition} \label{prop:dist1}
If $K_n=K_n(p_1,\ldots,p_n)$ is a pretzel knot with $n\geq 4$,  and $P$ is a $(0,b_0(K))$-bridge surface for $K_n$, then $d(P,K_n)=1$.
\end{proposition}

\begin{proof} Let $K_n$ be a pretzel knot with $n\geq 4$.  Arrange $K_n$ in $S^3$ so that the rational tangles are all intersecting the plane in $S^3$ where $z=0$.  Let $P$ be the plane where $z=0$.  Choose one disk between the first two tangles on one side of $P$ and another on the other side between the third and fourth tangle, as in Figure~\ref{fig:pretzeldisks}.  These two disks are bridge disks for the knot.  If we thicken these disks by taking an $\epsilon$ neighborhood of each, and then consider the boundary of these neighborhoods, we obtain disks which bound essential simple closed curves in $P$. Since they are disjoint curves, the vertices they correspond to in the curve complex are adjacent. Hence, $d(P,K_n)\leq 1$.  If the distance is zero, we would have disks on opposite sides of $P$ which bound essential simple closed curves that are isotopic to each other, so we can isotope our disks to have the same boundary, and form a sphere.  This implies that our knot has at least two components, one on each side of the sphere, contradicting our assumption that $K_n$ is a knot. Hence, $d(P,K_n)=1$.\end{proof}

\begin{figure}
\begin{center}
\includegraphics[width=6in]{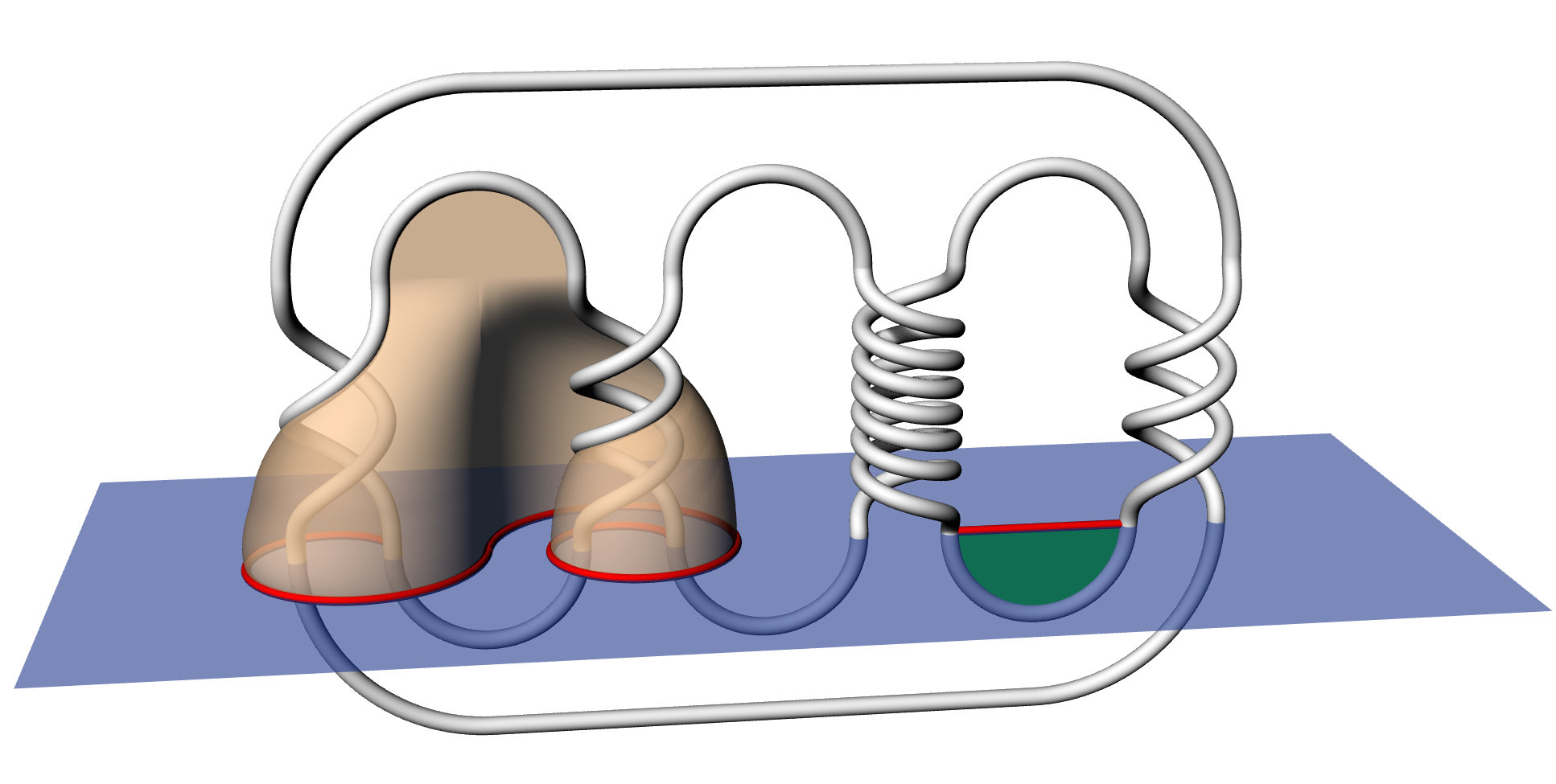}
\caption{A four branched pretzel knot with disjoint bridge disks on opposite sides of the bridge surface.}\label{fig:pretzeldisks}
\end{center}
\end{figure}

This gives us that any pretzel knot $K_n=K_n(p_1,\ldots,p_n)$ with $\gcd\{p_1,\ldots,p_n\}\not = 1$ and $n\geq4$ has a stair-step bridge spectrum and is not high distance.  Thus, one could not find the bridge spectrum of $K_n$ with Theorem \ref{thm:dist}.  This proves the next corollary.

\begin{corollary}\label{cor:SSbsNotHdist}
There exists knots with distance 1 and stair-step bridge spectrum. 
\end{corollary}

Thus, stair-step does not imply high distance.  As an example, $K_n(3p_1,3p_2,\ldots,3p_n)$ for $n\geq 4$ for any pretzel knot $K_n(p_1, \ldots,p_n)$ has ${\mathbf b}(K_n)=(n,n-1,\ldots,1,0)$ and distance one.

%Stairstep end
%%%%%%%%%%%%%%

\section{Bridge spectra of cables of 2-bridge knots}\label{sec:proof}

In this section we prove Theorem \ref{thm:BSofC2b}.  To do this, we compute the genus zero, one, and two bridge numbers of cables of 2-bridge knots.  The majority of the work here is establishing a lower bound for the genus one bridge number, which is done is subsection \ref{sec:genusone}.

\subsection{Genus zero and genus two bridge numbers}\label{sec:zeroandtwo}

From Schubert's and Schulten's result on bridge number of satellite knots, Theorem  \ref{thm:bridgesat}, we have the following.

\begin{lemma}\label{lem:genus0}
Let $K_{p/q}$ be a 2-bridge knot, $T_{m,n}$ a torus knot, and $K={\mathsf {cable}}(T_{m,n}, K_{p/q})$, then $b(K)=b_0(K)=2m$.
\end{lemma}

\begin{proof}
From Theorem \ref{thm:bridgesat}, we have $b_0(K)\geq 2m$.  To prove the upper bound, consider Figure \ref{fig:figure8Asurface}. If we were to cable $K_{p/q}$, we would produce a knot with $2m$ maxima, which gives us $b_0(K)\leq 2m$.\end{proof}

A similar lemma is required for the genus two bridge number.  We hope to relay some intuition by noticing that any cable of a two bridge knot embeds on a genus two surface.  

\begin{figure}[htbp]
\begin{center}
\includegraphics[height=2.2in]{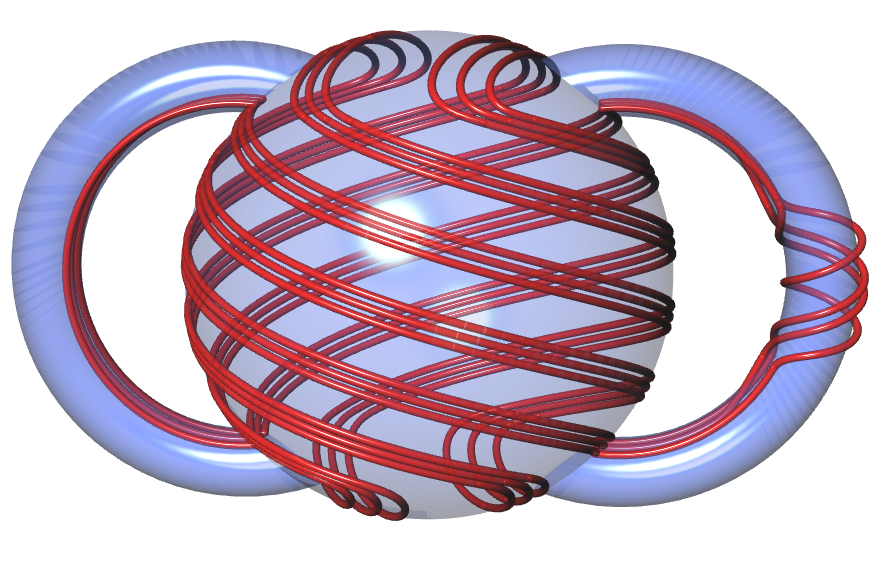}

\end{center}\caption{A cable of a 2-bridge knot embedded on a genus two surface}\label{fig:genus2}
\end{figure}

\begin{lemma}\label{lem:genus2}
Let $K$ be a cable of a non-torus 2-bridge knot by $T_{m,n}$. Then $b_2(K)=0$.
\end{lemma}

\begin{proof}
Recall Proposition \ref{prop:BSof2b}, which tells us that 2-bridge knots embed on a genus two surface.  From Figure \ref{fig:2-bridge} we see that this embedding can have a single arc of the knot on a handle, i.e. there is a meridian disk which intersects the knot once.  Thus, the remark following Lemma \ref{lem:Zlem6.1} tells us that the cable can embed in a genus two surface. See Figure \ref{fig:genus2}.\end{proof}

\subsection{Genus one bridge number}\label{sec:genusone}

\begin{theorem} \label{thm:BSofC2b}

Let $K_{p/q}$ be a non-torus 2-bridge knot and $T_{m,n}$ an $(m,n)$-torus knot. If $K:={\mathsf {cable}}(T_{m,n}, K_{p/q})$ is a cable of  $K_{p/q}$ by $T_{m,n}$, then the bridge spectrum of $K$ is ${\mathbf b}(K)=(2m,m,0)$.
\end{theorem}

\begin{proof}  
The previous two lemmas give us all but $b_1(K)$.   We observe from Figure \ref{fig:genus2} that eliminating either handle will produce an embedding with $m$ trivial arcs on either side of the torus.  Hence, $b_1(K)\leq m$.  For the rest of the proof, we need to show that $b_1(K)\geq m$.   In subsection \ref{sec:Operations}, we note that every bridge surface is either strongly irreducible or weakly reducible. This dichotomy gives us two main cases for the proof of the theorem.
    \begin{pcases}
      \case
{\em Suppose that $\Sigma$ is a strongly irreducible genus 1 bridge surface for $K$.}   Let $J$, also denote the knot $K$, let $\Sigma$ be the bridge surface, and let $S$ denote $T:=\partial E(K_{p/q})$.  With these choices, apply Lemma \ref{Z5.2} and get one of three situations, giving us the following subcases.\vspace{3mm}

\noindent {\em Subcase A: After isotopy, $\Sigma_K$ is transverse to $T$ and each component of $\Sigma_K \setminus \eta(T)$ is $\partial N(K)$-essential in $E(K) \setminus \eta(T)$.}  
\vspace{3mm}

The exterior $E(K)$ is split along $T$ into $E(K_{p/q})$ and $C_{m,n}$ and $\Sigma_K$ is $\partial N(K)$-essential in both. The cable space $C_{m,n}$ can be decomposed into $(T^2\times I)\cup V$, where $V=D^2\times S^1$, a solid torus.  By assumption, the bridge surface $\Sigma$ is transverse to K. 

{\em Claim:} The surface $\Sigma \cap V$ is a collection of meridian disks.  

To prove the claim, notice that $\Sigma \cap V$ is also essential in $V$ and cannot be a $\partial$-parallel annulus or $\partial$-parallel disk, as these surfaces are disjoint from $K$ or are $\partial N(K)$-$\partial$-compressible.  The only option left is that each component of $\Sigma \cap V$ is a meridian disk, proving the claim.  

\vspace{3mm}
The Claim gives us that $\p_-C_{m,n}\cap \Sigma$ is meridional, thus, by Lemma \ref{lem:Z3.3} we know that $\Sigma \cap T$ has meridian slope also.  This is also the boundary slope of $\Sigma \cap E(K_{p/q})$.  But by Theorem \ref{thm:HTthm}, no essential surface has a boundary slope of $1/0$, giving us a contradiction and completing this subcase.

\vspace{9mm}

\noindent {\em Subcase B: After isotopy, $\Sigma_K$ is transverse to $T$, one component of $\Sigma_K \setminus \eta(T)$ is $\partial N(K)$-strongly irreducible and all other components are $\partial N(K)$-essential in $E(K) \setminus \eta(T)$.}  Note that $\Sigma_K \setminus \eta(T)=[ \Sigma_K\cap E(K_{p/q}) ]\cup [\Sigma_K\cap C_{m,n}]$.

First assume that $\Sigma_K\cap C_{m,n}$ is $\partial N(K)$-strongly irreducible.  This means that $\Sigma_K\cap E(K_{p/q})$ is $\partial N(K)$-essential.  Again by Theorem \ref{thm:HTthm}, we know that $\Sigma_K\cap E(K_{p/q})$ is isotopic to an essential surface $S_n=S_n(n_1,n_2,\ldots,n_{k-1})$ and has negative Euler characteristic by Lemma \ref{lem:eulerneg}.  The bridge surface $\Sigma$ is a torus by assumption, and thus must have Euler characteristic zero.  We can only increase the Euler characteristic of $S_n$ by gluing on disks. Assume that the number of components of $\Sigma \cap T$ is minimal.  It must be that $\Sigma \cap T$ is a collection of essential, simple closed curves in $T$.  These curves must be the boundaries of disks in order for us to increase the Euler characteristic, but the only disks that have essential curves in $T$ as their boundary are meridian disks, which have boundary slope 1/0.  This gives a contradiction of Theorem \ref{thm:HTthm}.\vspace{3mm}

Now, assume that $ \Sigma_K\cap E(K_{p/q})$ is $\partial N(K)$-strongly irreducible and  $\Sigma_K\cap C_{m,n}$ is $\partial N(K)$-essential. This implies that $\Sigma_K\cap C_{m,n}$ is a collection of meridian disks, and a single merdian disk in $C_{m,n}$ will intersect $K$, or $\p_-C_{m,n}$, exactly $m$ times. Thus $|\Sigma \cap K|=m\cdot |\Sigma \cap \partial E(K_{p/q})|=m\cdot b'$, where $b':=|\Sigma \cap \partial E(K_{p/q})|$ and $\Sigma \cap \partial E(K_{p/q})$ is a $(1,b')$-bridge surface for $K_{p/q}$. Since $b_1(K_{p/q})=1$, then $b'\geq 1$ therefore $b_1(K)\geq m$, completing this subcase.\vspace{3mm}

\noindent {\em Subcase C: After isotopy, $\Sigma_K$ is almost transverse to $T$, and each component of $\Sigma_K \setminus \eta(T)$ is $\p N(K)$-essential in $M(J)\setminus \eta(T)$.} Any saddlepoint tangency in a genus one torus will produce a figure 8 shaped intersubsection, $c$. When cut along an open neighborhood of $c$, the torus splits into two disjoint components, one an annulus and the other a disk or a single disk.  In either case, the disk will be inessential, which gives a contradiction to $\Sigma_K \setminus \eta(T)$ being $\p N(K)$-essential, completing this Subcase and Case 1.

\vspace{5mm}

      \case[]      
{\em Suppose that $\Sigma$ is weakly reducible.}   We apply Theorem \ref{thm:HS}, which gives us that there exists a multiple bridge splitting $\{ \Sigma_0,S_1,\Sigma_1,\ldots,S_d,\Sigma_d\}$ of $(S^3,K)$, that is strongly irreducible and 

\begin{equation}
g(\Sigma)=\sum_{i=0}^d g(\Sigma_i) -\sum_{i=1}^d g(S_i).\tag{$\star$}
\end{equation}

Since $\Sigma$ is a genus one bridge surface, this sum must equal one. By Theorem \ref{thm:cableMeridianSmall}, $K$ is meridionally small since $K$ is the cable of a meridionally small knot, see Definition \ref{def:meridianallySmall}. Hence $E(K)$ contains no essential meridional surfaces, and so $S_i\cap K=\emptyset$ for all $i=1,\ldots,d$. Thus, by Lemma \ref{lem:Z3.5}, we have that each $S_i$ is isotopic to $T$. If we have a multiple bridge splitting  $\{\Sigma'_0, S'_1, \Sigma'_1,S'_2,\Sigma'_2 \}$ with $S'_1$ and $S'_2$ isotopic to $T$, then  $\Sigma'_0\sqcup_{S'_1} \Sigma'_1$ or $\Sigma'_1\sqcup_{S'_2} \Sigma'_2$ would be isotopic to $\Sigma'_0$ or  $\Sigma'_2$.  Therefore, we can assume without loss of generality that $d=1$ and our multiple bridge splitting is  $\{\Sigma_0, S_1, \Sigma_1\}$.  This is a strongly irreducible multiple bridge splitting and one of $\Sigma_0$ or $\Sigma_1$ is contained in the cable space $C_{m,n}$ and the other is contained in the exterior of the 2-bridge knot, $E(K_{p/q})$.  Let $\Sigma_0\subset C_{m.n}$ and $\Sigma_1\subset E(K_{p/q})$. As $K$ is not contained in $E(K_{p/q})$, $\Sigma_1$ must be a Heegaard surface.  Since the tunnel number of a 2-bridge knot is 1, by Proposition \ref{prop:MSYsprop}, then $g(\Sigma_1)=g(E(K_{p/q}))=2$.  This is because the genus of a 3-manifold is the minimum genus of a Heegaard splitting.  We also have $g(S_1)=1$, and thus by equation $(\star )$, $$1=2+g(\Sigma_0)-1.$$

This tells us that $g(\Sigma_0)=0$, but $\Sigma_0$ needs to be a bridge surface for $C_{m,n}$ and thus, a Heegaard surface, which gives a contradiction as the Heegaard genus of $C_{m,n}$ is greater than zero, since $C_{m,n}$ is not $S^3$, the only 3-manifold with a genus zero Heegaard splitting.  This completes the final case of the proof.\end{pcases}\end{proof}

%%%%%%%%%%%%%%%%%%%%%%%%%%%%%%%%%%%%%%

\section{Generalizations and Conjectures}\label{ch:generalize}

Montesinos knots are $n$ rational tangles connected as in Figure \ref{fig:mont}.  The 2-bridge knots are the special case in which  $n=1$ or 2.  A nontorus 2-bridge knot has primitive bridge spectrum $(2,1,0)$.  This generalizes to some Montesinos knots, as we see from Theorem \ref{thm:bsgmk}.  A special class of Montesinos knots are $n$-pretzel knots, having $\beta_i=1$ for all $i$; these also then have stair-step primitive bridge spectra if $\gcd\{\alpha_i\} \not = 1$. But these knots embed on a genus $n-1$ torus, see Figure \ref{fig:pretzelOnSurface}. Hence, Corollary \ref{cor:bsp} states that their bridge spectrum is $(n,n-1,\ldots,3,2,0)$. 

\begin{questionss}
Are pretzel knots exactly the subclass of Montesinos knots which have the property that ${\mathbf {\hat b}}(K)\not={\mathbf {b}}(K)$?
\end{questionss}

So, as with Theorem \ref{thm:BSofC2b}, we would like to know what happens to bridge spectra under the operation of cabling.  In general, we do not know. But following are examples of cabling that fail to produce an $m$-stair-step bridge spectrum.

Consider the pretzel knot, $K_4=K_4(2,-3,-3,3)$, as in Figure \ref{fig:pretzelOnSurface}, and cable $K_4$ by $T_{m,n}$, and let $K={\mathsf {cable}}(T_{m,n},K_4)$. Each time the knot $K_4$ passes to the back side of the genus 3 torus, we must introduce extra crossings to negate the twisting that we pick up when we force $K$ to lie on the same surface. See the following four illustrations in Figure \ref{fig:4images}.  Image A is obtained by blackboard framed cabling of a single strand.  Pulling this cable around all the way around the back to the front again gives us the images B and C.  Finally, if we were to let the crossings in the upper section of the cable pass through the surface to the front, we get the fourth figure. Notice that if the knot were not on the surface, all these crossings would cancel out and we would have $m$ strands with no crossings.  Hence, we have actually not introduced any new crossings into our knot $K$ with this change and all four images in Figure \ref{fig:4images} represent the same tangle.

%%%%%%%%%%%%%%%%%%%%%%%%%% Now we send the curves around the torus.

    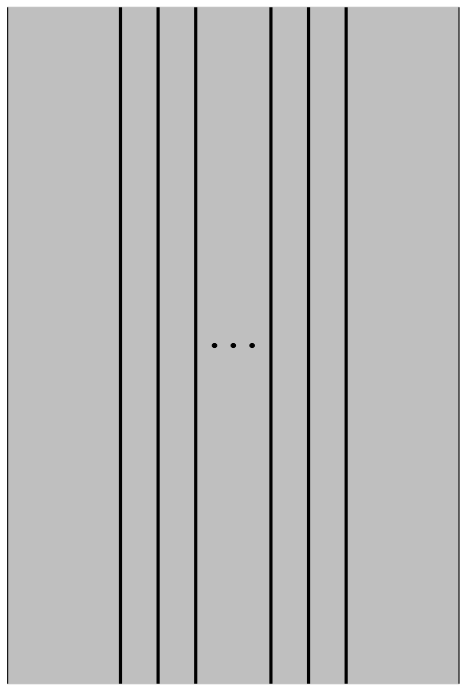
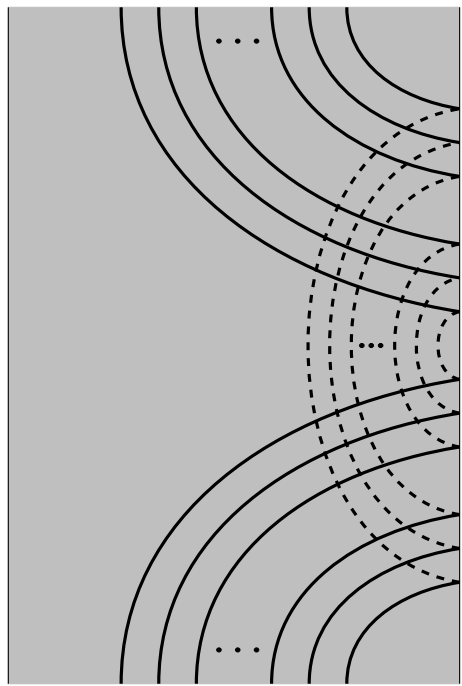
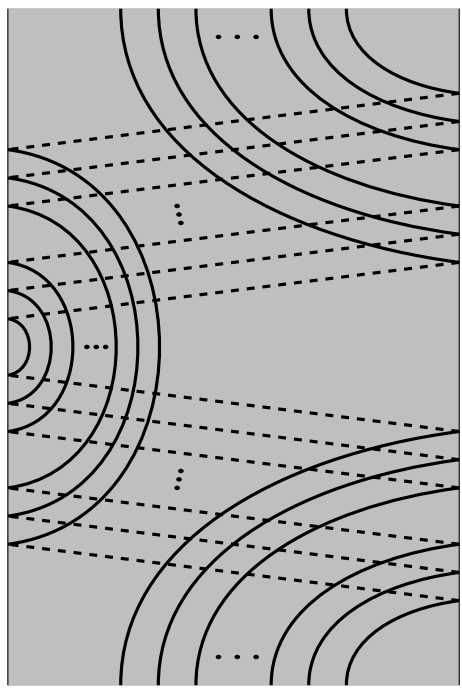
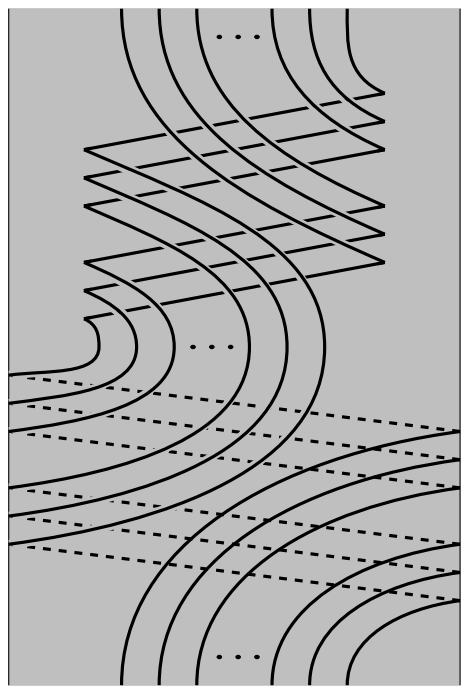
\begin{figure*}
        \centering
        \begin{subfigure}[b]{0.475\textwidth}
            \centering
\begin{tikzpicture}[yscale=.9,xscale=-1]

\draw [fill=lightgray,lightgray] (0,0) rectangle (6,10);

\draw (0,10) to (0,0);
\draw (6,10) to (6,0);

\foreach \b in {0,.25,.5}{%dots
	\draw [fill] (2.75+\b,5) circle [radius=0.03];
}

\foreach \a in {1.5,2,2.5,3.5,4,4.5}
	\draw [very thick] (\a,10) to (\a,0);

\end{tikzpicture}
            \caption[Network2]%
            {{\small $m$ strands on cylindrical section of the torus.}}    
            \label{fig:mean and std of net14}
        \end{subfigure}
        \hfill
        \begin{subfigure}[b]{0.475\textwidth}  
            \centering 
\begin{tikzpicture}[yscale=-.9,xscale=-1]

\draw [fill=lightgray,lightgray] (0,0) rectangle (6,10);

\draw (0,10) to (0,0);
\draw (6,10) to (6,0);

\foreach \a in {1.5,2,2.5,3.5,4,4.5}{
	\draw [very thick] (\a,10) to [out=270, in=10] (0,10-\a);
	\draw [very thick] (\a,0) to [out=90, in=-10] (0,\a);
	\draw [very thick, dashed] (0,4+\a) to [out=-10, in=10] (0,6-\a);
}

\foreach \b in {0,.25,.5}{%dots
	\draw [fill] (2.7+\b,9.5) circle [radius=0.03];
	\draw [fill] (2.7+\b,.5) circle [radius=0.03];
}
\foreach \b in {0,.125,.25}{%dots
	\draw [fill] (1.05+\b,5) circle [radius=0.03];
}

\end{tikzpicture}
            \caption[]%
            {{\small Pulling all the strands around to the back.}}    
            \label{fig:mean and std of net24}
        \end{subfigure}
        \vskip\baselineskip
        \begin{subfigure}[b]{0.475\textwidth}   
            \centering 
\begin{tikzpicture}[yscale=.75,xscale=-1]

\draw [fill=lightgray,lightgray] (0,0) rectangle (6,12);

\draw (0,12) to (0,0);
\draw (6,12) to (6,0);

\foreach \a in {1.5,2,2.5,3.5,4,4.5}{
	\draw [very thick] (\a,12) to [out=270, in=10] (0,12-\a);
	\draw [very thick] (\a,0) to [out=90, in=-10] (0,\a);
	\draw [very thick, dashed] (0,6+\a) to (6,5+\a);
	\draw [very thick, dashed] (0,6-\a) to (6,7-\a);
	\draw [very thick] (6,5+\a) to [out=190, in=170] (6,7-\a);

}

\foreach \b in {0,.25,.5}{%dots
	\draw [fill] (2.7+\b,11.5) circle [radius=0.03];
	\draw [fill] (2.7+\b,.5) circle [radius=0.03];
}

\foreach \b in {0,.125,.25}{%dots
	\draw [fill] (4.7+\b,6) circle [radius=0.03];
}
\foreach \b in {0,.25,.5}{%dots
	\draw [fill] (3.7+.1*\b,8.2+.6*\b) circle [radius=0.03];
	\draw [fill] (3.7+.1*\b,3.8-.6*\b) circle [radius=0.03];
}

\end{tikzpicture}
            \caption[]%
            {{\small Continue pulling the strands all the way to the front again.}}    
            \label{fig:mean and std of net34}
        \end{subfigure}
        \quad
        \begin{subfigure}[b]{0.475\textwidth}   
            \centering 
\begin{tikzpicture}[yscale=-.75,xscale=-1]

\draw [fill=lightgray,lightgray] (0,0) rectangle (6,12);

\draw (0,12) to (0,0);
\draw (6,12) to (6,0);

\foreach \a in {1.5,2,2.5,3.5,4,4.5}{
	\draw [very thick] (\a,12) to [out=270, in=10] (0,12-\a);
	\draw [very thick, dashed] (0,6+\a) to (6,5+\a);
	\draw [very thick] (1,6-\a) to (5,7-\a);

}
\foreach \a in {1.5,2,2.5,3.5,4,4.5}{	
	\draw [ultra thick,lightgray] (\a+.04,0) to [out=90, in=-30] (1.15,\a-.05);
	\draw [ultra thick,lightgray] (\a-.04,0) to [out=90, in=-30] (1,\a-.04);

	\draw [ultra thick,lightgray] (5.95,5.04+\a) to [out=190,in=90] (6.295-\a,6) to [out=270, in=150] (4.82,6.96-\a+.1);
	\draw [ultra thick,lightgray] (6,4.96+\a) to [out=190,in=90] (6.35-\a,6) to [out=270, in=150] (5.05,7.04-\a);

	\draw [very thick] (\a,0) to [out=90, in=-30] (1,\a);
	\draw [very thick] (6,5+\a) to [out=190,in=90] (6.3-\a,6) to [out=270, in=150] (5,7-\a);

}

\foreach \a in {1.5,2,2.5,3.5,4,4.5}{	
	\draw [very thick] (1,6-\a) to (1.1,6.025-\a);
	\draw [very thick] (5,7-\a) to (4.9,6.975-\a);
	
}

\foreach \b in {0,.25,.5}{%dots
	\draw [fill] (2.7+\b,11.5) circle [radius=0.03];
	\draw [fill] (2.7+\b,.5) circle [radius=0.03];
		\draw [fill] (3.05+\b,6) circle [radius=0.03];

}

\end{tikzpicture}
            \caption[]%
            {{\small Let the strands on the back pass to the front. }}    
            \label{fig:mean and std of net44}
        \end{subfigure}
        \caption[ The average and standard deviation of critical parameters ]
        {\small Four figures which illustrate how to cable pretzel knots, which introduces extra twisting, as seen in the upper half of Figure D.} 
        \label{fig:4images}
    \end{figure*}

By our convention, the standard torus knot $T_{m,n}$ can be described as a braid on $m$ strands. The corresponding braid word is $(\sigma_{1}^{-1} \sigma_{2}^{-1}\cdots\sigma_{m-2}^{-1}\sigma_{m-1}^{-1})^n$, where $\sigma_i$ is the $i$-th strand crossing over the $i+1$-th strand going from top to the bottom as in Figure \ref{fig:braidgen} and \ref{fig:torusbraid}.  One should observe that in a pretzel knot, for each positive crossing, we get a twist region as in Figure \ref{fig:4images} D and corresponding braid word for the upper half is $(\sigma_{m-1} \sigma_{m-2}\cdots\sigma_{2}\sigma_{1})^m$.  This will have the effect of adding a region of twists that is $p$ copies of this twist region, where $p=\sum_{i=1}^jp_i$, i.e. $(\sigma_{m-1} \sigma_{m-2}\cdots\sigma_{2}\sigma_{1})^{p\cdot m}$.

For example, to obtain a knot which we know the bridge spectrum, let $J_4=J_4(6,-9,-9,9)$, which triples the number of crossings in each twist region of $K_4$.  Then, by Corollary \ref{cor:bsp}, as $\alpha=\text{gcd}(6,-9,-9,9)=3 \not = 1$, we have ${\mathbf {\hat b}}(J_4)=(4,3,2,1,0)$.  Then we see that $p=6-9-9+9=-3$, where a negative will represent the inverse braid word.

So, if we were to cable $J_4$ by $T_{2,-7}$, we will have $K={\mathsf {cable}}(T_{2,-7}, J_4)$ with $(\sigma_{1})^{-3\cdot2}=\sigma_1^{-6}$ crossings we are forced to introduce to make $K$ lie on the genus 3 surface.  Taking this along with the fact that there are crossings corresponding to $( \sigma_{1}^{-1})^{-7}=\sigma_1^7$ from the torus knot, we have $ \sigma_{1}^{-6}\sigma_1^{7}=\sigma_1$.  Thus,  all the crossings that we introduced by forcing the knot to lie in the surface cancel out with crossings from $T_{2,-7}$ and one crossing from the torus knot is left over.  Hence, the bridge spectrum of $K$, ${\bf b}(K)\leq (8, 6, 4, 1)$ instead of the $(8,6,4,2)$ we might have expected by cabling with an index 2 torus knot. Note that $b_1$ and $b_2$ are only upper bounds also, as we have not proved they are the minimum value.  We formalize this observation in the following conjecture.

\begin{figure}[h]
\begin{tikzpicture}[yscale=.5]

\foreach \b in {0,.25,.5}{%dots
	\draw [fill] (2.75+\b,1.5) circle [radius=0.03];
}
\draw [very thick] (2,3) [out=270,in=90]  to (1.5,0);

\draw [ultra thick,white] (1.55,3) [out=270,in=90] to (2.05,0);
\draw [ultra thick,white] (1.45,3) [out=270,in=90] to (1.95,0);

\draw [very thick] (1.5,3) [out=270,in=90] to (2,0);

\foreach \a in {2.5,3.5,4}
	\draw [very thick] (\a,3) to (\a,0);

\end{tikzpicture}

\caption{The braid group generator $\sigma_1$ in the braid group $B_m$}\label{fig:braidgen}
\end{figure}
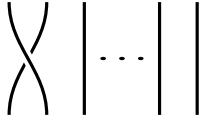

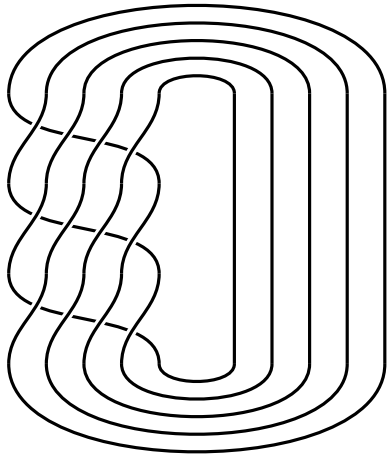
\begin{figure}[h]
\begin{tikzpicture}[yscale=.8]

\foreach \b in {0,1.5,3}{
\draw [very thick] (1.5,\b) [out=270,in=90] to (3.5,\b-1.5);

\foreach \a in {2,2.5,3,3.5}{

	\draw [ultra thick,white] (\a-.04,\b) [out=270,in=90] to (\a-.54,\b-1.5);
	\draw [ultra thick,white] (\a+.04,\b) [out=270,in=90] to (\a-.46,\b-1.5);
	
	\draw [very thick] (\a,\b) [out=270,in=90] to (\a-.5,\b-1.5);

}

}

\foreach \a in {1.5,2,2.5,3,3.5}{
\draw [very thick] (\a,3) [out=90,in=90] to (8-\a,3);
\draw [very thick] (\a,-1.5) [out=-90,in=-90] to (8-\a,-1.5);
\draw [very thick] (8-\a,-1.5) to (8-\a,3);

}

\end{tikzpicture}

\caption{The torus knot $T_{5,3}$ positioned in braid form, which has braid word $(\sigma_{1}^{-1} \sigma_{2}^{-1}\sigma_{3}^{-1}\sigma_{4}^{-1})^3$}\label{fig:torusbraid}
\end{figure}

\begin{conjecturess}\label{conj:pretzel1}
Let $K_j=K_j(p_1,p_2,\ldots,p_j)$ be a pretzel knot with stair-step primitive bridge spectrum. Let $T_{m,n}$ an $(m,n)$-torus knot and $p=\sum_{i=1}^jp_i$.  If $K={\mathsf {cable}}(T_{m,n}, K_j)$ is a $T_{m,n}$ cable of $K_j$, then the primitive bridge spectrum of $K$ is $${\mathbf {\hat b}}(K)=(mj,m(j-1),\ldots, 3m,2m, \min\{m,\vert mp-n\vert\},0).$$
\end{conjecturess}

We conjecture that the degeneration we see in the last nonzero bridge number in Conjecture \ref{conj:pretzel1} is analogous to the degeneration Zupan observed in iterated torus knots, see \cite{Z1}.  But when we consider to 2-bridge knots, we have no degeneration, which in this analogy, perhaps corresponds to Montesinos knots which are {\em not} pretzel.  Thus, this leads us to Conjecture \ref{conj:BSofMontandPret}, which stems from question 3.5 in \cite{O}.

A note on the direction of the possible proof for Conjecture \ref{conj:pretzel1}: It seems that methods similar to the ones used in this dissertation in the proof of the Theorem \ref{thm:BSofC2b} might work for $n=3$, but not for $n\geq 4$.  This is a consequence of Oertel's work in \cite{oertel} where he shows that Montesinos knots with four or more branches have closed incompressible surfaces in their exterior.  This stands in sharp contrast to the situation for 2-bridge knots, which played a key role in our proof of Theorem \ref{thm:BSofC2b}.

\begin{conjecturess}\label{conj:BSofMontandPret}
The bridge spectrum of an $(m,n)$-cable of a Montesinos knot $M_0$ is $m$-stair-step if and only if $M_0$ is not a pretzel knot. \end{conjecturess}

There is less direct evidence in support of this conjecture at this time.  In a similar vein, though, is the following conjecture.  We noticed that cabling is mostly well behaved in the cases we have studied.  There exists some degeneration at the end of the bridge spectrum, but we do not believe there can be degeneration elsewhere.   It would be of great interest to find counterexamples or a proof to the following conjecture.

\begin{conjecturess}\label{conj:cablingNice}
Let $J$ be a knot with bridge spectrum ${\mathbf b}(J)=(b_0(J),\ldots, b_g(J),0)$ and a cable $K={\mathsf {cable}}(T_{m,n}, J)$. Then $${\mathbf b}(K)=\left(m\cdot b_0(J),m\cdot b_1(J),\ldots,m\cdot b_{g-1}(J), m\cdot b_g(J),b_{g+1}(K),0\right).$$
\end{conjecturess}

%% Backmatter *****************************************************************************************

%\iffalse
%% Appendices

\appendix

%\section{Title} - Will display as Appendix A

%\fi

%%  Bibliography

\bibliography{thesisbib}
%\bibstyle{alpha}
%\nocite{*}

\end{document}